\documentclass[11pt,oneside]{amsart}
\usepackage{amsmath,ifthen, amsfonts, amssymb,
srcltx,
   amsopn, textcomp}

\usepackage{hyperref}
\usepackage{graphicx}
\usepackage{xypic}

\usepackage[small]{caption}
\usepackage{marginnote}

\newtheorem{thm}{Theorem}[section]
\newtheorem{lem}[thm]{Lemma}

\newtheorem{cor}[thm]{Corollary}

\newtheorem{prop}[thm]{Proposition}

\newtheorem*{claim}{Claim}
\newtheorem*{mainA}{Theorem~\ref{mainA}}

\newtheorem*{mainC}{Theorem~\ref{mainC}}

\newtheorem*{KeyProp}{Proposition~\ref{prop:inheritFiniteDim}}

\theoremstyle{definition}
\newtheorem{defn}[thm]{Definition}

\newtheorem{exmp}[thm]{Example}

\newcommand{\neb}{\mathcal N}

\DeclareMathOperator{\Aut}{Aut}

\newcommand{\homology}{\ensuremath{{\sf{H}}}}

\newcommand{\dist}{\textup{\textsf{d}}}

\newcommand{\field}[1]{\mathbb{#1}}
\newcommand{\integers}{\ensuremath{\field{Z}}}

\newcommand{\reals}{\ensuremath{\field{R}}}

\newcommand{\wt}[1]{\widetilde{#1}}

\newcommand{\la}{\langle}
\newcommand{\ra}{\rangle}

\DeclareMathOperator{\Isom}{Isom}

 \label{mainB}

\DeclareMathOperator{\stab}{Stab}

\newcommand{\hull}{\text{\sf hull}}

\begin{document}
\title{Classifying Virtually Special Tubular Groups}
\author{Daniel J. Woodhouse}
\email{daniel.woodhouse@mail.mcgill.ca}

\begin{abstract}
  A group is tubular if it acts on a tree with $\mathbb{Z}^2$ vertex stabilizers and $\mathbb{Z}$ edge stabilizers.
  We prove that a tubular group is virtually special if and only if it acts freely on a locally finite CAT(0) cube complex.
  Furthermore, we prove that if a tubular group acts freely on a finite dimensional CAT(0) cube complex, then it virtually acts freely on a three dimensional CAT(0) cube complex.
\end{abstract}

\maketitle

\section{Introduction}

A \emph{tubular group} $G$ splits as a graph of groups with $\mathbb{Z}^2$ vertex groups and $\mathbb{Z}$ edge groups.
Equivalently, $G$ is the fundamental group of a graph of spaces, denoted by $X$, with each vertex space homeomorphic to a torus and each edge space homeomorphic to $S^1\times [-1,1]$.
The graph of spaces $X$ is a \emph{tubular space}.
In this paper all tubular groups will be finitely generated and therefore have compact tubular spaces.

Tubular groups have been studied from various persectives:
Brady and Bridson provided tubular groups with isoperimetric function $n^{\alpha}$ for all $\alpha$ in a dense subset of $[2, \infty)$ in \cite{BradyBridson00}.
Cashen determined when two tubular groups are quasi-isometric \cite{Cashen10}.
Wise determined whether or not a tubular group acts freely on a CAT(0) cube complex \cite{Wise13}, and classified which tubular groups are cocompactly cubulated.
The author determined a criterion for finite dimensional cubulations \cite{Woodhouse14}.
Button has proven that all free-by-cyclic groups that are also tubular groups act freely on finite dimensional cube complexes~\cite{Button}.
The main theorem of this paper is:

\begin{mainA}
 A tubular group $G$ acts freely on a locally finite CAT(0) cube complex if and only if $G$ is virtually special.
\end{mainA}

Haglund and Wise introduced special cube complexes in \cite{HaglundWise08}.
The main consequence of a group being special is that it embeds in a right angled Artin group (see \cite{Wise09} or \cite{Wise12} for a full outline of Wise's program).

\subsection{Structure of the Paper}

In~\cite{Wise13} Wise obtained free actions of tubular groups on CAT(0) cube complexes by first finding \emph{equitable sets} that allow the construction of \emph{immersed walls}.
Such a set of immersed walls determines a \emph{wallspace} $(\wt{X},\mathcal{W})$ which yields a dual cube complex $C(\wt{X},\mathcal{W})$ which $G$ acts freely on.
Wallspaces were first introduced by Haglund and Paulin \cite{HaglundPaulin98}, and the dual cube complex construction was first developed by Sageev \cite{Sageev95}.
In~\cite{Woodhouse14} the author defined a criterion, called \emph{dilation}, that determines if an immersed wall produces infinite or finite dimensional cubulations.
More precisely, if the immersed walls are \emph{non-dilated}, then $C(\wt{X},\mathcal{W})$ is finite dimensional.
We recall the relevant definitions and background in Section~\ref{section:TubularGroups}.

Section~\ref{section:primitive} establishes a technical result using techniques from~\cite{Woodhouse14}.
It is shown that that immersed walls can be replaced with \emph{primitive} immersed walls without losing the finite dimensionality or local finiteness of the associated dual cube complex.
The reader is encouraged to either read this section alongside~\cite{Woodhouse14}, or skip it on a first reading.

In Section~\ref{section:FiniteDimensional} we analyse $C(\wt{X},\mathcal{W})$ in the finite dimensional case to establish a set of conditions that imply that $G \backslash C(\wt{X},\mathcal{W})$ is virtually special.
We decompose $C(\wt{X},\mathcal{W})$ as a tree of spaces, with the same underlying tree as $\wt{X}$ and then, under the assumption that the walls are \emph{primitive}, we show that $C(\wt{X},\mathcal{W})$ maps $G$-equivariantly into $\mathbf{R}^d \times \wt{\Gamma}$, where $\mathbf{R}$ is the standard cubulation of $\reals$, and $\wt{\Gamma}$ is the underlying graph of $\wt{X}$.
A further criterion, the notion of a \emph{fortified} immersed wall, determines when $C(\wt{X},\mathcal{W})$ is locally finite.
Combining these results allow us to give criterion for $G \backslash C(\wt{X},\mathcal{W})$ to be virtually special.

In Section~\ref{section:RevisitingEquitable} we consider a tubular group acting freely on a CAT(0) cube complex $\wt{Y}$.
We show that we can obtain from such an action immersed walls that preserve the important properties of $\wt{Y}$.
More precisely, we prove the following:

\begin{KeyProp}
 Let $G$ be a tubular group acting freely on a CAT(0) cube complex $\wt{Y}$.
 Then there is a tubular space $X$ and a finite set of immersed walls in $X$.
 Moreover, if $(\wt{X}, \mathcal{W})$ is the associated wallspace, then
 \begin{enumerate}
    \item $G$ acts freely on $C(\wt{X}, \mathcal{W})$.
 	\item \label{claim:finDim}  $C(\wt{X}, \mathcal{W})$ is finite dimensional if $\wt{Y}$ is finite dimensional.
 	\item \label{claim:locFin} $C(\wt{X}, \mathcal{W})$ is finite dimensional and locally finite if $\wt{Y}$ is locally finite.
 \end{enumerate}
\end{KeyProp}

\noindent This Proposition is sufficient to allow us to prove~\ref{mainA}.

In Section~\ref{section:VirtualCubicalDimension} we further exploit the results obtained in Section~\ref{section:FiniteDimensional} to obtain the following, demonstrating that the cubical dimension of tubular groups with finite dimensional cubulations are virtually within $1$ of their cohomological dimension.

 \begin{mainC}
  A tubular group acting freely on a finite dimensional CAT(0) cube complex has a finite index subgroup that acts freely on a 3-dimensional CAT(0) cube complex.
 \end{mainC}

%
%
%

\noindent \textbf{Acknowledgements:} I would like to thank Dani Wise and Mark Hagen.

\section{Background Tubular Groups and their Cubulations} \label{section:TubularGroups}

Let $G$ be a tubular group with associated tubular space $X$ and underlying graph $\Gamma$.
Given an edge $e$ in a graph we will let $-e$ and $+e$ respectively denote the initial and terminal vertices of $e$.
Let $X_v$ and $X_e$ denote vertex and edge spaces in this graph of spaces.
Let ${X}_{-e}$ and ${X}_{+e}$ be the boundary circles of $X_e$, and denote the attaching maps by $\varphi_e^{-} : {X}_{e}^{-} \rightarrow X_{-e}$, and $\varphi_e^{+} : {X}_e^{+} \rightarrow X_{+e}$.
Note that $\varphi_e^{-}$ and $\varphi_e^{+}$ respectively represent generators of $G_e$ in $G_{{e}^{-}}$ and $G_{e^{+}}$.
We will let $\wt{X}$ denote the universal cover of $X$.
Let $\widetilde{X}_{\widetilde{v}}$ and $\widetilde{X}_{\widetilde{e}}$ denote vertex and edge spaces in the universal cover $\wt{X}$, and let $\wt{\Gamma}$ denote the Bass-Serre tree.
We will assume that each vertex space has the structure of a nonpositively curved geodesic metric space and that attaching maps $\varphi_e^{-}$ and $\varphi_e^{+}$ define locally geodesic curves in $X_{-e}$ and $X_{+e}$.

\subsection{Equitable Sets and Intersection Numbers}

Given a pair of closed curves in a torus $\alpha , \beta : S^1 \rightarrow T$, the \emph{intersection points} are the elements $(p,q) \in S^1 \times S^1$ such that $\alpha(p) = \beta(q)$.
For a pair of homotopy classes $[\alpha], [\beta]$ of closed curves in a torus $T$, their \emph{geometric intersection number} $\# \big[[\alpha], [\beta] \big]$ is the minimal number of intersection points realised by a pair of representatives from the respective classes.
This number is realised by any pair of geodesic representatives of the classes.
If $B = \{ [\beta_i] \}$ is a finite set of homotopy classes of curves in $T$, then $\#[\alpha, B] := \sum_i \#[ \alpha, \beta_i ]$.
Viewing $[\alpha]$ and $[\beta]$ as elements of $\pi_1T = \mathbb{Z}^2$, we can compute that $\#\big[[\alpha], [\beta]\big] = \det \big[[\alpha], [\beta]\big]$.
Given an identification of $\mathbb{Z}^2$ with $\pi_1T$, the elements of $\mathbb{Z}^2$ are identified with homotopy classes of curves in $T$, so it makes sense to consider their geometric intersection number.
An \emph{equitable set} for a tubular group $G$ is a collection of sets $\{ S_v \}_{v \in \Gamma}$, where $S_v$ is a finite set of distinct geodesic curves in $X_v$ disjoint from the attaching maps of adjacent edge spaces, such that $S_v$ generate a finite index subgroup of $\pi_1X_v = G_v$, and $\#\big[ \varphi_e^{-}, S_{{e}^{-}}\big] = \#\big[ {\varphi}_e^{+}, S_{{e}^+}]$.
Note that equitable sets can also be given with $S_v$ a finite subset of $G_v$ that generates a finite index subgroup of $G_v$ and satisfies the corresponding equality for intersection numbers.
This is how Wise formulates equitable sets, and its equivalence follows from exchanging elements of $G_v = \pi_1X_v$ with geodesic closed curves in $X_v$ that represent the corresponding elements.
An equitable set is \emph{fortified} if for each edge $e$ in $\Gamma$, there exists $\alpha^+_e \in S_{+e}$ and $\alpha^-_e \in S_{-e}$ such that $\#[ \alpha^+_e , \varphi^+_e] = \#[\alpha^-_e, \varphi^-_e] = 0$.
An equitable set is \emph{primitive} if every element $\alpha \in S_v$ represents a primitive element in $G_v$.

\subsection{Immersed Walls From Equitable Sets}

Immersed walls are constructed from \emph{circles} and \emph{arcs}.
For each $\alpha \in S_v$, let $S^1_{\alpha}$ be the domain of $\alpha$.
The disjoint union $\bigsqcup S^1_\alpha$ over all $\alpha \in S_v$ and $v\in V\Gamma$ are the \emph{circles}.
Since $\#[ \varphi_e^- , S_{e^{-}}] = \#[ \varphi^+_e , S_{e^+}]$, there exists a bijection from the intersection points between curves in $S_{-e}$ and $\varphi_e^{-}$, and the intersection points between curves in $S_{+e}$ and $\varphi^+_e$.
 Let $(p^-, q^-)\in S^1_{\alpha^-} \times X^-_e$ and $(p^+, q^+) \in S^1_{\alpha^+}$ be corresponding intersection points between $\alpha^{\pm} \in S_{\pm e}$ and $\varphi^{\pm}_e$.
 Then an \emph{arc} $a \cong [0,1]$ has its endpoints attached to $p^-$ and $p^+$.
 The endpoints of $a$ are mapped into $X_{-e}\cap X_e$ and $X_{+e} \cap X_e$, so the interior of $a$ can be embedded in $X_e$.
 After attaching an arc for each pair of corresponding intersection points, we obtain a set $\{\Lambda_1, \ldots, \Lambda_n\}$ of connected graphs that map into $X$, called \emph{immersed walls}.
Each graph $\Lambda_i$ has its own graph of groups structure with infinite cyclic vertex groups and trivial edge groups.
{ \bf As in~\cite{Woodhouse14}, all ``immersed walls'' in this paper are immersed walls constructed from equitable sets as above. }
This means that we are free use the results obtained in~\cite{Woodhouse14}.

A lift of $\wt{\Lambda}_i \rightarrow \wt{X}$ is a two sided embedding in $\wt{X}$, separating $\wt{X}$ into two \emph{halfspaces}.
The images of the lifts of $\wt{\Lambda}_i$ to $\wt{X}$ are \emph{horizontal walls} $\mathcal{W}_{\textsf{h}}$.
The \emph{vertical} walls $\mathcal{W}_{\textsf{v}}$ are obtained from the lifts of curves $\alpha_e: S^1 \rightarrow X_e$ given by the inclusion $S^1 \times \{0\} \hookrightarrow S^1 \times [-1, 1]$.
The set $\mathcal{W} = \mathcal{W}_{\textsf{h}} \sqcup \mathcal{W}_{\textsf{v}}$ of all horizonal and vertical walls gives a \emph{wallspace} $(\wt{X}, \mathcal{W})$ where the $G$-action on $\wt{X}$ also gives an action on $\mathcal{W}$.
The main theorem of \cite{Wise13} is that a tubular group acts freely on a CAT(0) cube complex if and only if there exists an equitable set.

A set of immersed walls is \emph{fortified} if they are obtained from a fortified equitable set.
A set of immersed walls is \emph{primitive} if they are obtained from a primitive equitable set.

Let $\wt{\Lambda}$ and $\wt{\Lambda}'$ be horizontal walls in $\wt{X}$.
An point $x \in \wt{\Lambda} \cap \wt{\Lambda}$ is a \emph{regular intersection point} if it lies in a vertex space $\wt{X}_{\wt{v}}$, and the lines $\wt{\Lambda} \cap \wt{X}_{\wt{v}}$ and $\wt{\Lambda}' \cap \wt{X}_{\wt{v}}$ are non-parallel.
Otherwise, a point $x \in \wt{\Lambda} \cap \wt{\Lambda}'$ is a \emph{non-regular intersection point}, and either $x \in \wt{X}_{\wt{v}}$ where $\wt{\Lambda} \cap \wt{X}_{\wt{v}} = \wt{\Lambda}' \cap \wt{X}_{\wt{v}}$, or $x \in \wt{X}_{\wt{e}}$.

An \emph{infinite cube} in a CAT(0) cube complex $\wt{Y}$, is an sequence of $c_0, c_1, \ldots, c_n, \ldots$ such that $c_n$ is an $n$-cube in $\wt{Y}$, and $c_n$ is a face of $c_{n+1}$.
In~\cite{Woodhouse14} a \emph{dilation function} is constructed for each immersed wall $R: \pi_1 \Lambda \rightarrow \mathbb{Q}^*$, and an immersed wall is said to be \emph{dilated} if $R$ has infinite image.
The following is Thm 1.2 from that paper:

\begin{thm} \label{mainTheoremFrom15}
 Let $X$ be tubular space, and $(\wt{X}, \mathcal{W})$ the wallspace obtained from a finite set of immersed walls in $X$. The following are equivalent: 
 \begin{enumerate}
  \item \label{infDimensional} The dual cube complex $C(\wt{X},\mathcal{W})$ is infinite dimensional.
  \item \label{infCube} The dual cube complex $C(\wt{X},\mathcal{W})$ contains an infinite cube.
  \item One of the immersed walls is dilated.
 \end{enumerate}
\end{thm}

The following result is also obtained from~\cite{Woodhouse14} by combining Thm 1.2, Prop 4.6, and Prop 3.4.
The last part follows from the last paragraph of the proof of Prop 3.4.

\begin{prop} \label{prop:InfDimInfCube}
 Let $X$ be tubular space, and $(\wt{X}, \mathcal{W})$ the wallspace obtained from a finite set of immersed walls in $X$.
 If $C(\wt{X}, \mathcal{W})$ is infinite dimensional, then $\mathcal{W}$ contains an set of pairwise regularly intersecting walls of infinite cardinality, that correspond to the hyperplanes in an infinite cube in $C(\wt{X},\mathcal{W})$.
 Moreover, the infinite cube contains a canonical $0$-cube.
\end{prop}

\section{Primitive Immersed Walls} \label{section:primitive}

The following result uses the techniques in Section 5. of~\cite{Woodhouse14} to compute the dilation function.
Let $\Lambda$ be an immersed wall in $X$, and let $R: \pi_1(\Lambda) \rightarrow \mathbb{Q}^*$ be its dilation function.
If $R$ has finite image then $\Lambda$ is non-dilated.
Let $q: \Lambda \rightarrow \Omega$ be the quotient map obtained by crushing each circle to a vertex.
Note that the arcs in $\Omega$ correspond to the arcs in $\Lambda$.
The dilation function $R$ factors through $q_*: \pi_1 \Lambda \rightarrow \pi_1 \Omega$, so there exists a function $\hat{R}: \pi_1 \Omega \rightarrow \mathbb{Q}^*$ such that $R = \hat{R} \circ q_*$.
We can therefore determine if $\Lambda$ is dilated by computing the function $\hat{R}$.

We orient each arc in $\Lambda$ so that all arcs embedded in the same edge space are oriented in the same direction.
We orient the arcs in $\Omega$ accordingly.
We define a weighting $\omega: E(\Omega) \rightarrow \mathbb{Q}^*$.
Let $X_e$ be an edge space in $X$, and let $\sigma$ be an arc mapped into $X_e$ connecting the circles $C^-$ and $C^+$.
Let $\alpha^-: C^- \rightarrow X_{-e}$ and $\alpha^+: C^{+} \rightarrow X_{+e}$ be the corresponding elements in the equitable set.
Then $$\omega(\sigma) = \frac{\#[\varphi_e^-, \alpha^-]}{\#[\varphi_e^+, \alpha^+]}.$$

If $\gamma = \sigma_1^{\epsilon_1} \cdots \sigma_n^{\epsilon_n}$ is an edge path in $\Omega$ where $\sigma_i$ is an oriented arc in $\Omega$, and $\epsilon \in {\pm 1}$, then $\hat{R}(\gamma) = \omega(\sigma_1)^{\epsilon_1} \cdots \omega(\sigma_n)^{\epsilon_n}$.

\begin{lem} \label{lem:primitive}
	Let $X$ be a tubular space, and let $G =\pi_1 X$.
	Let $\Lambda_1, \ldots \Lambda_k$ be a set of immersed walls in $X$ obtained from an equitable set $\{S_v\}_{v\in\Gamma}$.
	Then there exists a set of primitive immersed walls $\Lambda_1',\ldots, \Lambda_{\ell}'$ in $X$ obtained from an equitable set $\{S_v'\}_{v \in \Gamma}$.
	Moreover:
    \begin{enumerate}
     \item If $\Lambda_1, \ldots, \Lambda_k$ are non-dilated, then so are $\Lambda_1',\ldots \Lambda_\ell'$.
     \item If $\Lambda_1, \ldots \Lambda_k$, are fortified, then so are $\Lambda_1',\ldots \Lambda_\ell'$.
     \end{enumerate}
\end{lem}
\begin{proof}
	Each $\Lambda_i$ decomposes as the union of disjoint circles, which are the domain of locally geodesic closed paths in the equitable set and arcs.
    Suppose that $\alpha^n \in S_v$, where $[\alpha] \in \pi_1 X_v = G_v$ is primitive.
    Let $\Lambda_i$ be the immersed wall containing the circle $S^1_{\alpha^n}$ corresponding to $\alpha^n$.

    A new equitable set is obtained by replacing $\alpha^n$ in $S_v$ with $n$ locally geodesic curves $\{ \alpha_i : S^1_i \rightarrow X_v\}_{i=1}^n$  with disjoint images in $X_v$ that are isotopic to $\alpha$ in $\wt{X}_v$.
    This remains an equitable set since $\#[\alpha^n, \gamma] = n\#[\alpha, \gamma] = \sum_{i=1}^n \#[\alpha_i, \gamma]$ for any locally geodesic curve $\gamma$ in $X_v$.
    New immersed walls are obtained from $\Lambda_i$ by replacing $S^1_{g^{n}}$ with $S^1_1, \ldots S^1_n$ and reattaching the arcs that were attached to the intersection points in $S^1_{\alpha^n}$ to the corresponding intersection points on $S^1_1, \ldots S^1_n$.
    Let $\Lambda_{i1},\ldots, \Lambda_{i\ell}$ be the new set of immersed walls obtained in this way.
    Note that each arc in $\Lambda_{i1},\ldots, \Lambda_{i\ell}$, corresponds to a unique arc in $\Lambda_i$.

    Assume that $\Lambda_i$ is non-dilated.
    We claim that the new immersed walls $\Lambda_{i1},\ldots, \Lambda_{i\ell}$ are also non-dilated.
    Let $q_i: \Lambda_i \rightarrow \Omega_i$ and $q_{ij}: \Lambda_{ij} \rightarrow \Omega_{ij}$ be the quotient maps obtained by crushing the circles to vertices.
    Let $u$ be the vertex in $\Omega_i$ corresponding to $S^1_{\alpha^n}$.
    Let $R_i : \pi_1 \Lambda_i \rightarrow \mathbb{Q}^*$ and $R_{ij}: \pi_1 \Lambda_{ij} \rightarrow \mathbb{Q}^*$ be the dilation functions.
    Let $\hat{R}_i : \pi_1 \Omega_i \rightarrow \mathbb{Q}^*$ and $\hat{R}_{ij} : \pi_1 \Omega_{ij} \rightarrow \mathbb{Q}^*$ be the unique maps such that $R_i = \hat{R}_i \circ q_i$ and $R_{ij} = \hat{R}_{ij} \circ q_{ij}$.
    Let $\omega_i$ and $\omega_{ij}$ be the respective weightings of the arcs in $\Omega_i$ and $\Omega_{ij}$.
    By assumption, $R_i$ and $\hat{R}_i$ have finite image.
    As the arcs in $\Lambda_{ij}$ correspond to arcs in $\Lambda_{i}$, there is a map $\rho_i : \Omega_{ij} \rightarrow \Omega_i$.
    We show $\Lambda_{ij}$ is non-dilated by showing that $\hat{R}_{ij}(\gamma) = \hat{R}_i(\rho_i \circ \gamma)$.

    Let $\sigma$ be an oriented arc in $\Omega_{ij}$.
    The edge $q_{ij}^{-1}(\sigma)$ embeds in an edge space $X_e$.
    If the vertices of $\sigma^{\epsilon}$ are disjoint from $\rho^{-1}(u)$, then $\omega_{ij}(\sigma^{\epsilon}) = \omega_i(\rho_{ij} \circ \sigma^{\epsilon})$.
    If the endpoints of $\sigma^{\epsilon}$ are contained in $\rho_{ij}^{-1}(v)$, and correspond to the circles $S^1_\iota$ and $S_\tau^1$ then
    $$\omega_{ij}(\sigma^{\pm 1}) = \frac{\#[\varphi^{\mp}_e, \alpha_\iota]}{\#[\varphi^{\pm}_e, \alpha_\tau]} = \frac{n\#[\varphi^{\mp}_e, \alpha_\iota]}{n\#[\varphi^{\pm}_e, \alpha_\tau]} = \frac{\#[\varphi^{\mp}_e, \alpha^n]}{\#[\varphi^{\pm}_e, \alpha^n]} = \omega(\rho_{ij} \circ \sigma^{\pm 1}). $$

    Suppose that exactly one endpoint of $\sigma^{\epsilon}$ is contained in $\rho^{-1}_{ij}(u)$. If $\sigma^{\epsilon}$ terminates a vertex in $\rho_{ij}^{-1}(u)$ corresponding to $S^1_\tau$, and the initial vertex corresponds to a circle that is the domain of a locally geodesic curve $\beta$ then
    $$\omega_{ij}(\sigma^{\pm 1}) = \frac{\#[\varphi^{\mp}_e, \beta]}{\#[\varphi^{\pm}_e, \alpha_\tau]} = \frac{n\#[\varphi^{\mp}_e, \beta]}{n\#[\varphi^{\pm}_e, \alpha_\tau]} = \frac{n\#[\varphi^{\mp}_e, \beta]}{\#[\varphi^{\pm}_e, \alpha^n]} = n \omega(\rho_{ij} \circ \sigma^{\pm 1}). $$
    \noindent If $\sigma^{\epsilon}$ starts at a vertex in $\rho_{ij}^{-1}(u)$ corresponding to $S^1_\iota$, and the terminal vertex correspond to a circle that is the domain of a locally geodesic curve $\beta$ then
    $$\omega_{ij}(\sigma^{\pm 1}) = \frac{\#[\varphi^{\mp}_e, \alpha_\iota]}{\#[\varphi^{\pm}_e, \beta]} = \frac{n\#[\varphi^{\mp}_e,\alpha_\iota ]}{n\#[\varphi^{\pm}_e, \beta]} = \frac{\#[\varphi^{\mp}_e, \alpha^n]}{n\#[\varphi^{\pm}_e, \beta]} = \frac{1}{n} \omega(\rho_{ij} \circ \sigma^{\pm 1}). $$
    \noindent Therefore, given an edge path $\gamma$ in $\Omega_{ij}$, since the number of edges exiting vertices in $\rho^{-1}_{ij}(v)$ is the same as the number of vertices entering, $\hat{R}_{ij}(\gamma) = \hat{R}_i(\rho_i \circ \gamma)$.

    This procedure produces immersed walls with one fewer non-primitive element in the equitable set.	
    Repeating this procedure for each non-primitive element in the equitable set produces a primitive set of non-dilated immersed walls.
    It is also clear, that if $\Lambda_1, \ldots, \Lambda_k$ are fortified, then so are the new immersed walls.
\end{proof}

\section{Finite Dimensional Dual Cube Complexes} \label{section:FiniteDimensional}


Let $X$ be a tubular space and let $G = \pi_1 X$.
Let $(\wt{X}, \mathcal{W})$ be the wallspace obtained from a set of non-dilated immersed walls $\Lambda_1, \ldots \Lambda_k$ constructed from an equitable set, and a vertical immersed wall in each edge space.
{\bf We emphasize that in this section all immersed walls are assumed to be non-dilated, even when it is not explicitly stated. }
Let $\wt{Z} = C(\wt{X}, \mathcal{W})$ and let $Z = G \backslash \wt{Z}$.
By Theorem~\ref{mainTheoremFrom15}, the immersed walls being non-dilated is equivalent to $\wt{Z}$ being finite dimensional.
   For each edge $\wt{e}$ in $\wt{\Gamma}$ let $\wt{\Lambda}_{\wt{e}}$ denote the vertical wall in $\wt{X}_{\wt{e}}$.

We refer to~\cite{WiseHruska13} for full background on the dual cube complex construction.
A $0$-cube $z$ in $\wt{Z}$ is a choice of halfspace $z[\wt{\Lambda}]$ of $\wt{\Lambda}$ for each $\wt{\Lambda} \in \mathcal{W}$ such that
\begin{enumerate}
  \item If $\wt{\Lambda}_1, \wt{\Lambda}_2 \in \mathcal{W}$, then $z[\wt{\Lambda}_1] \cap z[\wt{\Lambda}_2] \neq \emptyset$.
  \item If $x \in X$, then there are only finitely many $\wt{\Lambda} \in \mathcal{W}$ such that $x \notin z[\wt{\Lambda}]$.
\end{enumerate}
\noindent Two $0$-cubes $z_1,z_2$ are adjacent if $z_1[\wt{\Lambda}] = z_2[\wt{\Lambda}]$ for all but precisely one hyperplane $\wt{\Lambda}_{12} \in \mathcal{W}$.
The $1$-cube joining $z_1$ and $z_2$ is dual to the hyperplane corresponding to $\wt{\Lambda}_{12}$.
An $n$-cube is then present wherever the $1$-skeleton of an $n$-cube appears.
We say that two disjoint walls $\wt{\Lambda}, \wt{\Lambda}'$ \emph{face} each other in $z$ if $z[\wt{\Lambda}]$ is not contained in $z[\wt{\Lambda}']$ and vice versa.

\begin{prop} \label{prop:graph of spaces}
There is a $G$-equivariant map $f: \wt{Z} \rightarrow \wt{\Gamma}$.
Therefore $\wt{Z}$ decomposes as a tree of spaces with $\wt{Z}_{\wt{v}} = f^{-1}(\wt{v})$, and $\wt{Z}_{\wt{e}} = f^{-1}(\wt{e})$ is the carrier of the hyperplane corresponding to $\wt{\Lambda}_{\wt{e}} \in \mathcal{W}_{\textsf{v}}$.
\end{prop}

\noindent By $f^{-1}(e)$ we mean the union of all cubes $c$ in $\wt{Z}$ such that $f(c) = e$.

\begin{proof}
    As there is a vertical wall in each edge space, and since the vertical walls are all disjoint we can identify $C(\wt{X}, \mathcal{W}_v)$ with the Bass-Serre tree $\wt{\Gamma}$ of $\wt{X}$.
    We define a map $f: \wt{Z} \rightarrow \wt{\Gamma}$: let $z$ be a $0$-cube in $Z$, then define $f(z)$ by letting $f(z)[\wt \Lambda_e] = z[\wt \Lambda_e]$.
    If $z_1, z_2$ are adjacent $0$-cubes, then $z_1[\wt \Lambda] \neq z_2[\wt \Lambda]$ for precisely one wall $\wt{\Lambda} \in \mathcal{W}$.
    If $\wt \Lambda$ is a horizontal wall then $f(z_1) = f(z_2)$ and the $1$-cube joining them is also mapped to the same vertex.
    If $\wt \Lambda = \wt \Lambda_e \in \mathcal{W}_{\textsf{v}}$ then $f(z_1)$ and $f(z_2)$ are adjacent in $\wt{\Gamma}$, so the $1$-cube joining $z_1$ and $z_2$ maps to the edge joining $f(z_1)$ and $f(z_2)$.
    As $f$ is defined on the $1$-skeleton, the map extends uniquely to the entire cube complex $\wt{Z}$.
     Then $\wt{Z}_{\wt v} = f^{-1}(\wt{v})$ and $\wt{Z}_{\wt{e}} =f^{-1}(\wt{e})$ is the carrier of the hyperplane corresponding to $\wt{\Lambda}_e$.
\end{proof}

\noindent Proposition~\ref{prop:graph of spaces} implies that $Z$ decomposes as a graph of spaces with vertex spaces $Z_v$, edge spaces $Z_e$, and underlying graph $\Gamma = G \backslash \wt{\Gamma}$.

The following proposition, which collects the principal consequences of finite dimensionality, is Prop 4.12 in \cite{Woodhouse14}.

\begin{prop} \label{prop:partition}
 Let $X$ be tubular space with geodesic attaching maps, and let $(\wt{X}, \mathcal{W})$ be the wallspace obtained from a finite set of immersed walls in $X$.
  If the dual cube complex $C(\wt{X},\mathcal{W})$ is finite dimensional, then the horizontal walls in $\mathcal{W}$ can be partitioned into a collection $\mathcal{P}$ of subsets such that:
  \begin{enumerate}
   \item \label{part:1} The partition $\mathcal{P}$ is preserved by $G$,
   \item \label{part:2} For each $A\in \mathcal{P}$, the walls in $A$ are pairwise non-intersecting.
   \item \label{part:4} Let $\wt{\Lambda} \in A \in \mathcal{P}$ be a wall intersecting $\wt{X}_{\wt{v}}$. There exists $h \in G_{\wt{v}}$ stabilizing an axis in $\wt{X}_{\wt{v}}$ perpendicular to $\wt{\Lambda} \cap \wt{X}_{\wt{v}}$ such that $A = \{ h^r \wt{\Lambda} \}_{r\in \mathbb{Z}}$.
  \end{enumerate}
\end{prop}

\noindent Any partition of the horizontal walls in $\mathcal{W}$ satisfying conditions (\ref{part:1})-(\ref{part:4}) in Proposition~\ref{prop:partition} will be called a \emph{stable partition}.

\begin{lem} \label{part:3}
Let $X$ be a tubular space and $(\wt{X}, \mathcal{W})$ be the wallspace obtained from a finite set of immersed walls in $X$.
Let $\mathcal{P}$ be a stable partition of the horizontal walls in $\mathcal{W}$.
Then for each $\wt{v} \in \wt{\Gamma}$, only finitely many $A \in \mathcal{P}$ contain walls intersecting $\wt{X}_{\wt{v}}$.
\end{lem}

\begin{proof}
 Suppose that $\wt{\Lambda}$ is a wall intersecting $\wt{X}_{\wt{v}}$, then, by condition (\ref{part:4}) of a stable partition, there exists some $h \in G_{\wt{v}}$ that is perpendicular to $\wt{\Lambda} \cap \wt{X}_{\wt{v}}$ such that $\{ h^r \wt{\Lambda} \}_{r\in \mathbb{Z}} \in \mathcal{P}$.
 By $G$-invariance we can deduce that each of the $G_{\wt{v}}$-translates of $\{ h^r \wt{\Lambda} \}_{r\in \mathbb{Z}}$ is also in $\mathcal{P}$.
 There are only finitely many such translates, therefore each $G_{\wt{v}}$-orbit of a wall in $\wt{X}_{\wt{v}}$ is contained in finitely many elements of $\mathcal{P}$.
 The claim then follows from the fact that there are only finitely many $G_{\wt{v}}$-orbits of walls intersecting $\wt{X}_{\wt{v}}$.
\end{proof}

The immersed walls $\Lambda_1, \ldots, \Lambda_k$ are non-dilated, and therefore $\wt{Z}$ is finite dimensional, so by Proposition~\ref{prop:partition} there exists a stable partition $\mathcal{P}$ of the horizontal walls in $\mathcal{W}$.
	Let $\mathcal{P}_{\wt{v}}$ be the  subpartition containing walls intersecting $\wt{X}_{\wt{v}}$.
	Let $\mathcal{P}_{\wt{e}}$ be the  subpartition of walls intersecting $\wt{X}_{\wt{e}}$.
    By Lemma~\ref{part:3}, both $\mathcal{P}_{\wt{v}}$ and $\mathcal{P}_{\wt{e}}$ are finite subpartitions.
	If $\wt{e}$ is incident to the vertex $\wt{v}$, then $\mathcal{P}_{\wt{e}} \subseteq \mathcal{P}_{\wt{v}}$.
    Let $\mathcal{P}_{\wt{v}} = \{ A_1, \ldots, A_{d_{\wt{v}}} \}$.
    By criterion~(\ref{part:4}) of a stable partition $A_i = \{h^r_i \wt{\Lambda}_i \}_{r\in \integers}$ such that $h_i \in G_{\wt{v}}$ stablizes an axis in $\wt{X}_{\wt{v}}$ perpendicular to $\wt{\Lambda}_i \cap \wt{X}_{\wt{v}}$.
    The action of $G_{\wt{v}}$ preserves both the partition $\mathcal{P}_{\wt{v}}$ and the ordering of the walls in each $A_i$.

    Let $\mathbf{R}$ denote the cubulation of $\mathbb{R}$ with a vertex for each integer and an edge joining consecutive integers.
Therefore, each $0$-cube in $\mathbf{R}^d$ is an element of $\mathbb{Z}^d$.
    We construct a free action of $G_{\wt{v}}$ on $\mathbf{R}^{d_{\wt{v}}}$.
     Let $g\in G_{\wt{v}}$ and let $(\alpha_1, \ldots, \alpha_{d_{\wt{v}}})$ be a $0$-cube in $\mathbf{R}^{d_{\wt{v}}}$.
      Define the map $g \cdot (\alpha_1, \ldots, \alpha_{d_{\wt{v}}}) = (\beta_1, \ldots, \beta_{d_{\wt{v}}})$ such that $g \cdot h_i^{\alpha_i} \wt{\Lambda}_i = h_j^{\beta_j} \wt{\Lambda}_j$.
      As $g$ permutes the walls in $\mathcal{P}_{\wt{v}}$, the map $g$ is a bijection on the $0$-cubes in $\mathbf{R}^{d_{\wt{v}}}$.
     If $g \cdot h_i^{\alpha_i} \wt{\Lambda}_i = h_j^{\beta_j} \wt{\Lambda}_j$, then necessarily $g \cdot h_i^{\alpha_i+1} \wt{\Lambda}_i = h_j^{\beta_j \pm 1} \wt{\Lambda}_j$, so adjacent $0$-cubes are mapped to adjacent $0$-cubes and the map extends to an isomorphism of $\mathbf{R}^{d_{\wt{v}}}$.
     If $g\cdot (\alpha_1, \ldots, \alpha_{d_{\wt{v}}}) = (\alpha_1, \ldots, \alpha_{d_{\wt{v}}})$ then $g$ would stabilize all the walls in $\mathcal{P}_{\wt{v}}$, which would imply that it fixed every $0$-cube in $\wt{Z}_{\wt{v}}$.
     Since $G_{\wt{v}}$ acts freely on $\wt{Z}_{\wt{v}}$ this would imply that $g = 1_G$, and hence $G_{\wt{v}}$ acts freely on $\mathbf{R}^{d_{\wt{v}}}$.

     We also define an embedding $\phi_{\wt{v}}: \wt{Z}_{\wt{v}}\rightarrow \mathbf{R}^{d_{\wt{v}}}$.
    If $z$ is a $0$-cube in $\wt{Z}_{\wt{v}}$ then every wall $\wt{\Lambda}$ that is either vertical or not in contained in the subpartition $\mathcal{P}_{\wt{v}}$ has $\wt{X}_{\wt{v}} \subseteq z[\wt \Lambda ]$.
    Therefore $z$ is entirely determined by $z[\wt \Lambda]$ for $\wt \Lambda$ in $\mathcal{P}_{\wt v}$.
    For $1 \leq i \leq d\wt v$ the set $\{ h^r_i \wt{\Lambda}_i \cap \wt{X}_{\wt{v}}\}_{r\in \mathbb{Z}}$ is an infinite collection of disjoint parallel lines in $\wt{X}_{\wt{v}}$.
    As all the walls in $A_i$ are disjoint in $\wt{X}$, for each $0$-cube $z$ in $\wt{Z}_{\wt{v}}$ there exists a unique $\alpha_i \in \mathbb{Z}$ such that $h_i^{\alpha_i} \wt{\Lambda}_i$ and $h_i^{\alpha_i + 1} \wt{\Lambda}_i$ face each other in $z$.
    Let $\phi_{\wt{v}}(z) = (\alpha_1, \ldots, \alpha_{d_{\wt{v}}})$.
    Note that the map is injective and sends adjacent $0$-cubes to adjacent $0$-cubes, so the map on the $0$-cubes extends to an embedding of the entire cube complex.

    \begin{lem} \label{lem:GvEquivariance}
    The embedding $\phi_{\wt{v}} :\wt{Z}_{\wt{v}} \rightarrow \mathbf{R}^{d_{\wt{v}}}$ is $G_{\wt{v}}$-equivariant.
    \end{lem}

    \begin{proof}
    Let $g \in G_{\wt{v}}$.
    If $\phi_{\wt{v}}(z) = (\alpha_1, \ldots, \alpha_{d_{\wt{v}}})$ and $g\cdot \phi_{\wt{v}}(z) = (\beta_1, \ldots, \beta_{d_{\wt{v}}})$, then $(gz)[h_i^{\alpha_i} \wt{\Lambda}_i] = z[g h_i^{\alpha_i} \wt{\Lambda}_i] = z[h_j^{\beta_j}\wt{\Lambda}_j]$, which implies that $\phi_{\wt{v}}(gz) = (\beta_1, \ldots, \beta_{d_{\wt{v}}}) = g\phi_{\wt{v}}(z)$.
    \end{proof}

  Let $\wt{e}$ be an edge adjacent to $\wt{v}$.
  Then either $+\wt{e} = \wt{v}$ or $-\wt{e} = \wt{v}$.
  We define a free action of $G_{\wt{e}}$ on $\mathbf{R}^{d_{\wt{e}}} \times [-1,1]$.
  After reindexing, let $\mathcal{P}_{\wt{e}} = \{ A_1, \ldots, A_{d_{\wt{e}}} \} \subseteq \mathcal{P}_{\wt{v}}$ where $d_{\wt{e}} \leq d_{\wt{v}}$.
  Let $(\alpha_1, \ldots, \alpha_{d_{\wt{e}}}, \pm 1)$ be a $0$-cube in $\mathbf{R}^d \times [-1,1]$ and let $g \in G_{\wt{e}}$.
  Then $g \cdot (\alpha_1, \ldots, \alpha_{d_{\wt{e}}}, \pm 1) = (\beta_1, \ldots, \beta_{d_{\wt{e}}}, \pm 1)$ such that $g \cdot h_i^{\alpha_i} \wt{\Lambda}_i = h_j^{\beta_j} \wt{\Lambda}_j$.
  As in the case of vertex spaces, this map extends to an isomorphism of $\mathbf{R}^{d_{\wt{e}}} \times [-1,1]$.

 As with the vertex spaces, there is a $G_{\wt{e}}$-equivariant embedding $\phi_{\wt{e}}: \wt{Z}_{\wt{e}} \rightarrow \mathbf{R}^{d_{\wt{e}}} \times [-1,1]$.
 Let $z$ be a $0$-cube in $\wt{Z}_{\wt{e}}$.
 Then for each $1 \leq i \leq d_{\wt{e}}$ there exists a unique $\alpha_i$ such that $h_i^{\alpha_i} \wt{\Lambda}_i$ faces $h_i^{\alpha_i +1} \wt{\Lambda}_i$ in $z$, and $\wt{X}_{\pm\wt{e}} \subseteq z[\wt{\Lambda}_{\wt{e}}]$.
 Define $\phi_{\wt{e}}(z) = (\alpha_1, \ldots, \alpha_{d_{\wt{e}}}, \pm 1)$.

 Let $\wt v = \pm \wt e$.
 The free action of $G_{\wt{v}}$ on $\mathbf{R}^{d_{\wt{v}}}$ restricts to a free action of $G_{\wt{e}}$.
 We claim that we can embed $\mathbf{R}^{d_{\wt{e}}}\times \{\pm 1\}$ into $\mathbf{R}^{d_{\wt{v}}}$ in a $G_{\wt{e}}$-equivariant way.
 Let $H_{\wt{e}} \subseteq \wt{Z}$ be the hyperplane corresponding to $\wt{\Lambda}_e$.
 As $\wt{Z}_{\wt{e}}$ is the carrier of $H_{\wt{e}}$, we can identify $\wt{Z}_{\wt{e}}$ with $H_{\wt{e}} \times [-1,1]$.
 Note that $H_{\wt{e}} \times \{\pm 1\}$ embeds as a subspace in $\wt{Z}_{\wt{v}}$, and $\phi_{\wt{e}}$ restricts to an embedding $\phi_{\wt{e}}^{\pm}: H_{\wt{e}} \times \{\pm 1 \} \rightarrow \mathbf{R}^{d_{\wt{e}}}$, where $\wt{v} = \pm \wt{e}$.

  We construct an embedding $\Psi^{\pm}_{\wt{e}}: \mathbf{R}^{d_{\wt{e}}} \rightarrow \mathbf{R}^{d_{\wt{v}}}$.
  Recall that $\mathcal{P}_{\wt{e}} = \{A_1, \ldots, A_{d_{\wt{e}}} \} \subseteq \mathcal{P}_{\wt{v}} = \{ A_1, \ldots, A_{d_{\wt{v}}} \}$.
     For $d_{\wt{e}} < j \leq d_{\wt{v}}$ if $h_j^r\wt{\Lambda}_j \in A_j$, then $\wt{X}_{\wt{e}} \subseteq z[h_j^r\wt{\Lambda}_j]$ for all $0$-cubes $z$ in $\wt{Z}_{\wt{e}}$.
     Therefore, there is a unique $\alpha_j^{\wt{e}} \in \mathbb{Z}$ such that $h_j^{\alpha^{\wt{e}}_j} \wt{\Lambda}_j$ faces $h_j^{\alpha^{\wt{e}}_j+1} \wt{\Lambda}_j$ for every $0$-cube $z$ in $\wt{Z}_{\wt{e}}$ and $d_{\wt{e}} < j \leq d_{\wt{v}}$.
     Thus we define
     \[
        \Psi^{\pm}_{\wt{e}}(\alpha_1, \ldots, \alpha_{d_{\wt{e}}}) = (\alpha_1, \ldots, \alpha_{d_{\wt{e}}}, \alpha^{\wt{e}}_{d_{\wt{e}} + 1}, \ldots, \alpha^{\wt{e}}_{d_{\wt{v}}}).
     \]
    \noindent The $G_{\wt{e}}$-equivariance of $\Psi^{\pm}_{\wt{e}}$ will require a further assumption:

 \begin{lem} \label{lem:GeEquivariantEdges}
   The following commutative square is $G_{\wt{e}}$-equivariant provided the immersed walls are primitive.
 \begin{displaymath}
  \xymatrix{
    H_{\wt{e}} \times \{\pm 1\} \ar@{^{(}->}[rr]^{\phi_{\wt{e}}^{\pm}}  \ar@{^{(}->}[d] & & \mathbf{R}^{d_{\wt{e}}}  \ar[d]^{\Psi^{\pm}_{\wt{e}}} \\
   \wt{Z}_{\wt{v}}  \ar@{^{(}->}[rr]^{\phi_{\wt{v}}}  & & \mathbf{R}^{d_{\wt{v}}} \\
  }
 \end{displaymath}
 \noindent Moreover, $\Psi^{\pm}_{\wt{e}}$ is a $G_{\wt{e}}$-equivariant inclusion that is equivalent to extending the $G_{\wt{e}}$-action on $\mathbf{R}^{d_{\wt{e}}}$ by a trivial action on $\mathbf{R}^{d_{\wt{v}} - d_{\wt{e}}}$.
 \end{lem}

 \begin{proof}
  Let $z$ be a $0$-cube in $H_{\wt{e}} \times \{ \pm 1 \}$.
  Then by construction
  $$\Psi^{\pm}_{\wt{e}}\circ \phi_{\wt{e}}^{\pm} (z) = (\alpha_1, \ldots , \alpha_{d_{\wt{e}}}, \alpha_{d_{\wt{e}}+1}^{\wt{e}}, \ldots, \alpha_{d_{\wt{v}}}^{\wt{e}}) = \phi_{\wt{v}}(z).$$

  To verify that $\Psi_{\wt{e}}$ is $G_{\wt{e}}$-equivariant, let $g\in G_{\wt{e}}$.
  For $1 \leq i \leq d_{\wt{e}}$ there exists $1 \leq j \leq d_{\wt{e}}$ and $\beta_j$ be such that $g \cdot h_i^{\alpha_i} \wt{\Lambda}_i = h_j^{\beta_j} \wt{\Lambda}_j$.
  For $d_{\wt{e}} < i \leq d_{\wt{v}}$ the intersection $\wt{\Lambda}_i \cap \wt{X}_{\wt{v}}$ is a geodesic line parallel to $\wt{X}_{\wt{e}} \cap \wt{X}_{\wt{v}}$.
  Thus, $G_{\wt{e}}$ stabilizes $\wt{\Lambda}_i \cap \wt{X}_{\wt{v}}$.
  As the immersed walls are primitive we can deduce that $G_{\wt{e}}$ stabilizes $\wt{\Lambda}_i$.
  For $d_{\wt{e}} < i \leq d_{\wt{v}}$ we deduce that $g \cdot h_i^{\alpha} \wt{\Lambda}_i = h_i^{\alpha} \wt{\Lambda}_i$ for $\alpha \in \mathbb{Z}$, and conclude:
   \begin{align*} g \cdot \Psi_{\wt{e}}(\alpha_1, \ldots, \alpha_{\wt{e}}) &= g\cdot(\alpha_1, \ldots , \alpha_{d_{\wt{e}}}, \alpha_{d_{\wt{e}}+1}^{\wt{e}}, \ldots, \alpha_{d_{\wt{v}}}^{\wt{e}}) \\
   &= (\beta_1, \ldots , \beta_{d_{\wt{e}}}, \alpha_{d_{\wt{e}}+1}^{\wt{e}}, \ldots, \alpha_{d_{\wt{v}}}^{\wt{e}}) \\
   &= \Psi_{\wt{e}}(\beta_1, \ldots, \beta_{\wt{e}}) \\
   &= \Psi_{\wt{e}}(g\cdot(\alpha_1, \ldots, \alpha_{\wt{e}})).
   \end{align*}

   \noindent Observe that $G_{\wt{e}}$ acts trivially on the last $d_{\wt{v}} - d_{\wt{e}}$ coordinates.
%
 \end{proof}

 \noindent 
    Let $d = \max \{ |\mathcal{P}_{\wt{v}} | \mid \wt{v} \in V\wt{\Gamma} \}$, which is finite, since there are only finitely many vertex orbits.

\begin{prop} \label{prop:RxT}
 If the immersed walls $\Lambda_1, \ldots, \Lambda_k$ are primitive, then $G$ acts freely on $\mathbf{R}^d \times \wt{\Gamma}$ such that the action on the $\wt{\Gamma}$ factor is the action of $G$ on the Bass-Serre tree.
 Moreover, there is a $G$-equivariant embedding $\phi: \wt{Z} \rightarrow \mathbf{R}^d \times \wt{\Gamma}$.
\end{prop}

\begin{proof}
    The $G_{\wt{v}}$ and $G_{\wt{e}}$-actions on $\mathbf{R}^{d_{\wt{v}}}$ and $\mathbf{R}^{d_{\wt{e}}}$ can be equivariantly extended to actions on $\mathbf{R}^d$ such that $G_{\wt{v}}$ and $G_{\wt{e}}$ act trivially on the additional factors.
    Therefore, the $G_{\wt{e}}$-commutative square in Lemma~\ref{lem:GeEquivariantEdges} can be extended:

    \begin{displaymath}
  \xymatrix{
   H_{\wt{e}} \times \{\pm 1\} \ar@{^{(}->}[rr]^{\phi_{\wt{e}}^{\pm}}  \ar@{^{(}->}[d] & & \mathbf{R}^{d_{\wt{e}}} \ar[d]^{\Psi^{\pm}_{\wt{e}}} \ar[r] & \mathbf{R}^d \ar[d]   \\
   \wt{Z}_{\wt{v}}  \ar@{^{(}->}[rr]^{\phi_{\wt{v}}}  & & \mathbf{R}^{d_{\wt{v}}} \ar[r] & \mathbf{R}^d \\
  }
 \end{displaymath}
    \noindent Therefore, we obtain a $G$-action on $\mathbf{R}^d \times \wt{\Gamma}$ and a $G$-equivariant embedding of the tree of spaces $\wt{Z} \rightarrow \mathbf{R}^d \times \wt{\Gamma}$.
\end{proof}

\begin{prop} \label{prop:fortified}
  $\wt{Z}$ is locally finite if and only if ${\Lambda}_1, \ldots, {\Lambda}_k$ are fortified.
\end{prop}

\begin{proof}
    If $\Lambda_1, \ldots, \Lambda_k$ is not fortified, then there exists a vertex space $\wt{X}_{\wt{v}}$ and an adjacent edge space $\wt{X}_{\wt{e}}$ such that every horizontal wall $\wt \Lambda$ in $\mathcal{P}_{\wt{v}}$ intersects $\wt{X}_{\wt{v}}$ as a line $\wt{\Lambda} \cap \wt{X}_{\wt v}$ that intersects $\wt{X}_{\wt{v}} \cap \wt{X}_{\wt{e}}$.
    Therefore, every horizontal wall intersecting $\wt{X}_{\wt{v}}$ intersects $\wt{X}_{\wt{e}}$, so $\mathcal{P}_{\wt e} = \mathcal{P}_{\wt v}$.
    Let $\wt{e}_1 , \ldots, \wt{e}_i, \ldots$ be an enumeration of the $G_{\wt{v}}$-orbit of $\wt{e}$.
     Then $\mathcal{P}_{\wt{e}_i} = \mathcal{P}_{\wt v}$ and $\wt{\Lambda}_{\wt{e}_i}$ intersects all the horizontal walls in $\mathcal{P}_{\wt{v}}$.

    Let $z$ be a $0$-cube in $\wt{Z}_{\wt v}$.
     There is a $0$-cube $z_i$ such that $z_i[\wt{\Lambda}] = z[\wt{\Lambda}]$ for $\wt{\Lambda} \neq \wt{\Lambda}_{\wt{e}_i}$, and $z_i[\wt{\Lambda}_{\wt{e}_i}] \neq z[\wt{\Lambda}_{\wt{e}_i}]$.
     To verify $z_i$ is a $0$-cube note that every wall in $\mathcal{P}_{\wt{v}}$ intersects $\wt{\Lambda}_{\wt{e}}$, and every other wall $\wt{\Lambda}$ that is not $\wt{\Lambda}_e$ has $\wt{X}_{\wt{e}} \subseteq z_{i}[\wt{\Lambda}]$.
     Therefore $z_i[\wt{\Lambda}_{\wt{e}}] \cap z_i[\wt{\Lambda}] \neq \emptyset$ for all $\wt{\Lambda} \in \mathcal{W} - \{\wt{\Lambda}_{\wt{e}}$\}.
     For any walls $\wt{\Lambda}_1, \wt{\Lambda}_2 \in \mathcal{W} - \{\wt{\Lambda}_{\wt{e}} \}$ the intersection $z_i[\wt{\Lambda}_1] \cap z_i[\wt{\Lambda}_2] = z[\wt{\Lambda}_1] \cap z[\wt{\Lambda}_2] \neq \emptyset$.
     Finally, if $x \in \wt{X}$, then $x \in z_i[\wt{\Lambda}]$ for all but finitely many $\wt{\Lambda} \in \mathcal{W}$, because it is true for $z$, which differs from $z_i$ on precisely one wall.
    Each $z_i$ is adjacent to $z$ since they differ on precisely one wall, so $z_1, \ldots, z_i, \ldots$ is an infinite collection of distinct $0$-cubes adjacent to $z$, and $\wt{Z}$ is not locally finite.

    To show that converse, we first observe that the embedding $\phi_{\wt{v}}: \wt{Z}_{\wt{v}} \rightarrow \mathbf{R}^{d_{\wt{v}}}$ proves that $\wt{Z}_{\wt{v}}$ is always locally finite, irrespective of whether the immersed walls are fortified.
    Let $z$ be a $0$-cube in $\wt{Z}_{\wt{v}}$, and let $\wt{u}$ be adjacent to $\wt{v}$ in $\wt{\Gamma}$ via an edge $\wt{e}$.
    Then $z$ can be adjacent to at most one $0$-cube $z_{\wt{e}}$ in $\wt{Z}_{\wt{u}}$ such that $z[\wt{\Lambda}] = z_{\wt{e}}[\wt{\Lambda}]$ for all $\wt{\Lambda} \in \mathcal{W}$ except $\wt{\Lambda}_e$.
    This $z_{\wt{e}}$ may not always define a $0$-cube however.
    Let $\wt{e}$ be an edge adjacent to $\wt{v}$.
    As the immersed walls are fortified there exists $\{ h^r \wt{\Lambda} \}_{r\in \integers} \in \mathcal{P}_{\wt{v}}$ such that $\{h^r \wt{\Lambda} \cap \wt{X}_{\wt{v}}\}_{r\in \mathbb{Z}}$ is an infinite set of lines parallel to $\wt{X}_{\wt{e}} \cap \wt{X}_{\wt{v}}$.
    As $\{h^r \wt{\Lambda}\}$ is a set of disjoint walls, there exists $r$ such that $h^r\wt{\Lambda}$ and $h^{r+1}\wt{\Lambda}$ are facing in $z$.
    There are only finitely many edges $g_1\wt{e}, \ldots, g_m\wt{e} \in G_{\wt{v}}\wt{e}$ such that $\wt{X}_{g_i\wt{e}} \subseteq z[h^r\wt{\Lambda}] \cap z[h^{r+1}\wt{\Lambda}]$.
    If $g\wt{e}$ is an edge such that $\wt{X}_{g_i\wt{e}}$ is not contained in $z[h^r\wt{\Lambda}] \cap z[h^{r+1}\wt{\Lambda}]$, then either $z_{g\wt{e}}[h^r\wt{\Lambda}] \cap z_{g\wt{e}}[\wt{\Lambda}_e] = \emptyset$ or $z_{g\wt{e}}[h^{r+1}\wt{\Lambda}] \cap z_{g\wt{e}}[\wt{\Lambda}_e] = \emptyset$ so $z_{g\wt{e}}$ is not a $0$-cube.
    As there are only finitely many $G_{\wt{v}}$-orbits of edges incident to $\wt{v}$ we conclude that $\wt{z}_{\wt{e}}$ is a $0$-cube for finitely many edges $\wt{e}$ incident to $\wt{v}$.
    %
   \end{proof}

\begin{prop} \label{prop:virtually special}
If $\Lambda_1, \ldots, \Lambda_k$ are primitive, fortified, non-dilated immersed walls, then $G$ is virtually special.
\end{prop}

\begin{proof}
   By Proposition~\ref{prop:RxT}, There is a free action of $G$ on $\mathbf{R}^d \times \Aut([-1,1]^d)$, so $G$ is a subgroup of $\Isom(\mathbf{R}^d \times \wt{\Gamma}) \cong \big(\mathbb{Z}^d \rtimes \Aut([-1,1])^d \big) \times \Aut(\wt \Gamma)$.
    Therefore, there is a projection $\rho: G \rightarrow \mathbb{Z}^d \rtimes \Aut([-1,1]^d)$.
    Each vertex group $G_{\wt{v}}$ embeds in $\mathbb{Z}^d$, and the mapping is invariant under conjugation.
    As there are only finitely orbits of vertices in $\wt{\Gamma}$, there exists a finite index subgroup $(D\mathbb{Z})^d \leqslant \mathbb{Z}^d$ such that if $\wt{e}$ is incident to $\wt{v}$ then $G_{\wt{e}} \cap (D\mathbb{Z})^d$ is generated by a primitive element in $G_{\wt{v}} \cap (D\mathbb{Z})^d$.
    Let $G' = \rho^{-1}\big((D\mathbb{Z})^d \big)$.
    Then $G' \leqslant G$ is a finite index subgroup that embeds in $(D\mathbb{Z})^d \times \Aut(\wt{\Gamma})$ such that each edge group is generated by an element that is primitive in the adjacent vertex groups.
    By Proposition~\ref{prop:RxT} there is a $G$-equivariant embedding $\phi: \wt{Z} \rightarrow \mathbf{R}^d \times \wt{\Gamma}$.
    As $G'$ does not permute the factors of $\mathbf{R}^d$ we can deduce that the hyperplanes in $\big( \mathbf{R}^d \times \wt{\Gamma} \big) / G$ do not self-intersect, so neither do the hyperplanes in $\wt{Z}/ G'$.
    Indeed, they are also $2$-sided and cannot inter-osculate.

    Let $G''$ be a finite index subgroup such that the underlying graph $\Gamma''$ has girth at least $2$.
    Let $\wt{e}$ be an edge such that $+\wt{e} = \wt{v}$.
    As $\Lambda_1, \ldots, \Lambda_k$ are fortified, we conclude that $d_{\wt{v}} > d_{\wt{e}}$ and $\wt{Z}^+_{\wt{e}}$ is a proper subcomplex of $\wt{Z}_{\wt{v}}$.
    As $G_{\wt e}''$ is primitive, if $g \in G_{\wt{v}}'' - G_{\wt e}''$ then $g h_{d_{\wt{v}}}^r \wt{\Lambda}_{d_{\wt{v}}} \neq h_{d_{\wt{v}}}^r \wt{\Lambda}_{d_{\wt{v}}}$ as $g$ acts by translation on $\wt{X}_{\wt{v}}$ in a direction non-parallel to $\wt{\Lambda}_{d_{\wt{v}}} \cap \wt{X}_{\wt{v}}$.
    Thus, $\wt{Z}^{\pm}_{\wt{e}}$ is not stabilizes by $g$, so we can deduce that $\stab_{G''_{\wt{v}}}(\wt{Z}^+_{\wt{e}}) = G_{\wt{e}}''$.
    Therefore, $G''_{\wt{e}} \backslash \wt{Z}^+_{\wt{e}}$ embeds in $G''_{\wt{v}} \backslash \wt{Z}_{\wt{v}}$.

    Let $Z'' = G'' \backslash \wt{Z}$.
    Let $z$ be a $0$-cube in $Z_v''$.
    Let $H_e$ be the vertical hyperplane contained in $Z_e''$, and dual to an edge incident to $v$.
    As the attaching maps of $Z_e''$ are embeddings, and $\Gamma''$ has girth at least $s$, we deduce that $z$ can only be incident to one end of a single $1$-cube intersected by $H_e$.
    Therefore $H_e$ does not self-osculate.

    Let $\sigma$ be a $1$-cube in $\mathbf{R}^d \times \wt{\Gamma}$ that projects to a $1$-cube in $\mathbf{R}^d$.
    The $G''$-orbit of $\sigma$ is a set of $1$-cubes, that all project to the same factor of $\mathbf{R}^d$, since $G''$ does not permute the factors of $\mathbf{R}^d$.
    As $G''$ does not invert hyperplanes, after subdividing $\mathbf{R}^d$ we can assume that the $G''$-orbit of $\sigma$ is a disjoint set of $1$-cubes.
    Therefore, after the corresponding subdivision, we conclude that the horizontal hyperplanes in $Z''$ don't self-osculate.
\end{proof}

We note that the requirement in Proposition~\ref{prop:virtually special} that the immersed walls are fortified is necessary, as the following example demonstrates.

\begin{exmp}
 Let $G = \langle a,b,t \mid [a,b] = 1, tat^{-1} = a \rangle$.
 We can decompose $G$ as the cyclic HNN extension of the vertex group $G_v = \langle a,b \rangle$ with stable letter $t$.
 Thus, $G$ is a tubular group.
 Let $X$ be the corresponding tubular space with a single vertex space $X_v$ and edge space $X_e$.
 There is an equitable set $\{\alpha_1, \alpha_2\}$ where $\alpha_1$ is a  geodesic curve in $X_v$ representing $ab \in G_v$, and $\alpha_2$ is a geodesic curve in $X_v$ representing $ab^{-1} \in G_v$.
 Note that each attaching map $\varphi_e^+$ and $\varphi_e^-$ intersects each curve in the equitable set precisely once.
 Therefore, we obtain a pair of embedded immersed horizontal walls $\Lambda_1$ and $\Lambda_2$, by connecting respective intersection points with $\varphi_e^+$ and $\varphi_e^-$ by an arc.
 A vertical wall $\Lambda_e$ is also embedded in $X_e$.

 In the wallspace $(\wt{X}, \mathcal{W})$ we can decompose $\mathcal{W}$ into three sets of disjoint walls: the walls $\mathcal{W}_1$ that cover $\Lambda_1$, the walls $\mathcal{W}_2$ that cover $\Lambda_2$, and the walls $\mathcal{W}_e$ that cover $\Lambda_e$.
 These walls are disjoint since the immersed walls are embedded.
 Furthermore, the walls in different sets pairwise intersect.
 Therefore we can conclude that $C(\wt{X}, \mathcal{W}) = \mathbf{R}^2 \times \wt{\Gamma}$.
 As this is not locally finite, $G \backslash C(\wt{X},\mathcal{W})$ cannot be virtually special.
\end{exmp}

\section{Revisiting Equitable Sets} \label{section:RevisitingEquitable}

Although Wise proved in \cite{Wise13} that acting freely on a CAT(0) cube complex $\wt{Y}$ implied the existence of an equitable set, and thus a system of immersed walls as in Section~\ref{section:TubularGroups}, no relationship was established between $\wt{Y}$ and the resulting dual $C(\wt{X}, \mathcal{W})$.
Proposition~\ref{prop:inheritFiniteDim} gives the relationship required to reduce Theorem~\ref{mainA} to considering cubulations obtained from equitable sets.

This section will apply the following theorem from \cite{Woodhouse16}.
A \emph{cubical quasiline} is a CAT(0) cube complex quasi-isomorphic to $\mathbb{R}$.

\begin{thm} \label{thm:isometricTori}
Let $G$ be virtually $\mathbb{Z}^n$.
Suppose $G$ acts properly and without inversions on a CAT(0) cube complex $\wt{Y}$.
Then $G$ stabilizes a finite dimensional subcomplex $\wt{Z} \subseteq \wt{Y}$ that is isometrically embedded in the combinatorial metric, and $\wt{Z} \cong \prod_{i=1}^m C_i$, where each $C_i$ is a cubical quasiline and $m \geq n$.
Moreover, $\stab_G(\Lambda)$ is a codimension-1 subgroup for each hyperplane $\Lambda$ in $\wt{Z}$.
\end{thm}

\noindent Theorem~\ref{thm:isometricTori} allows us to prove the following:

\begin{lem} \label{lem:findingTori}
Let $G$ be a tubular group acting freely on a CAT(0) cube complex $\wt{Y}$.
Let $G_v$ be a vertex group in $G$, then there exists a $G_v$-equivariant subspace $\wt{X}_v \subseteq \wt{Y}$ homeomorphic to $\mathbb{R}^2$.
Moreover, $\wt{X}_{\wt{v}}$ has a metric such that the intersection of a hyperplane in $\wt{Y}$ with $\wt{X}_{v}$ is either empty, or a geodesic line.
\end{lem}

\begin{proof}
 By Theorem~\ref{thm:isometricTori}, there exists a $G_{v}$-equivariant subcomplex $\wt{Y}_v \subseteq \wt{Y}$  that isometrically embeds in the combinatorial metric, and such that $\wt{Y}_v \cong \prod_{i=1}^m C_i$ where each $C_i$ is a cubical quasiline.
 By the flat torus theorem~\cite{BridsonHaefliger}, $G_v$ stabilizes a flat $\wt{X}_v \subseteq \wt{Y}_v$, that is a convex subset in the CAT(0) metric of $\wt{Y}_v$.
  As the stabilizers of hyperplanes in $\wt{Y}_v$ are codimension-1 subgroups of $G_v$, the intersection of a hyperplane in $\wt{Y}$ with $\wt{X}_v$ is either empty or a geodesic line in the CAT(0) metric inherited from $\wt{Y}_v$.
\end{proof}

If $S$ is a subset of a CAT(0) cube complex $\wt{Y}$, then let $\hull(S)$ denote the \emph{combinatorial convex hull} of $S$.
The combinatorial convex hull of $S$ is the minimal convex subcomplex containing $S$.
Equivalently, $\hull(S)$ is the intersection of all closed halfspaces containing $S$.

\begin{defn}
 Let $X$ be a tubular space and let $Y$ be a nonpositively curved cube complex.
A map $f: X \rightarrow Y$ is an \emph{amicable immersion} if:
 \begin{enumerate}
 \item $f_* : \pi_1 X \rightarrow \pi_1 Y$ is an isomorphism.
 \item The $G$-equivariant map $\wt f : \wt{X} \rightarrow \wt{Y}$ embeds each vertex space $\wt{X}_{\wt{v}}$ in $\wt{Y}$.
 \item Each $\wt{X}_{\wt{v}}$ has a Euclidean metric such that if $H \subseteq \wt Y$ is a hyperplane, then the intersection $H \cap \wt{X}_{\wt v}$ is either the empty set, or a single geodesic line in $\wt{X}_{\wt v}$.
 \item Each edge space $\wt{X}_{\wt{e}}$ is emdedded transverse to the hyperplanes.
  \item Each $\wt{X}_{\wt{e}}$ is contained in $\hull\big( \bigcup_{\wt{v}\in V\wt{\Gamma}} \wt{X}_{\wt{v}}\big)$.
 \end{enumerate}
 \noindent Note that the Euclidean metric on each $\wt X_{\wt{v}}$ is not the subspace metric induced from $\wt{Y}$.
\end{defn}

\begin{lem}
 Let $X$ be a tubular space and let $Y$ be a nonpositively curved cube complex.
 Let $F : \pi_1 X \rightarrow \pi_1 Y$ be an isomorphism.
 Then there is an amicable immersion $f: X \rightarrow Y$ such that $f_* = F$.
\end{lem}

\begin{proof}
 Use $F$ to identify $G = \pi_1 X$ with $\pi_1 Y$.
 The claim is proven by constructing a $G$-equivariant map between the tree of spaces $\wt{X} \rightarrow \wt{Y}$.
 By Lemma~\ref{lem:findingTori} for each $\wt v \in V\wt{\Gamma}$, we can $G_{\wt{v}}$-equivariantly embed a Euclidean flat $\wt{X}_{\wt{v}}$ in $\wt{Y}$ such that if $H \subseteq \wt{Y}$ is a hyperplane, then the intersection $H \cap \wt{X}_{\wt v}$ is either the empty set, or a single geodesic line in $\wt{X}_{\wt v}$.
 Moreover, we can ensure that $\bigcup_{\wt{v} \in V\wt \Gamma} \wt{X}_{\wt{v}}$ is $G$-equivariant.
  The edges spaces $\wt{X}_{\wt{e}}$ can then be inserted transverse to the hyperplanes in $\wt{Y}$ so that the intersections $\wt{X}_{\wt{e}} \cap \wt{X}_{\wt{v}}$ with adjacent vertex spaces is not contained in a hyperplane in $\wt{Y}$, and $\wt{X}_{\wt{e}}$ is contained inside $\hull\big( \bigcup_{\wt{v}\in V\wt{\Gamma}} \wt{X}_{\wt{v}}\big)$.
\end{proof}


\begin{lem} \label{lem:flatHull}
Let $X \rightarrow Y$ be an amicable immersion, where $Y$ is finite dimensional.
If $\wt{v}$ is a vertex in $\wt{\Gamma}$, then $\hull(\wt{X}_{\wt{v}})$ embeds as a subcomplex of $\mathbf{R}^d$ for some $d$.
\end{lem}

\begin{proof}
Let $G = \pi_1 X$.
Let $y$ be a $0$-cube in $\hull(\wt{X}_{\wt{v}})$.
If $H$ is a hyperplane in $\wt{Y}$, let $y[H]$ denote the halfspace of $H$ containing $y$.
Each $0$-cube is determined by the halfspace containing it for each hyperplane.
If $H$ is a hyperplane that doesn't intersect $\wt{X}_{\wt{v}}$, then $y[H]$ is the halfspace containing $\wt{X}_{\wt{v}}$, and therefore is $y[H]$ fixed for all $0$-cubes $y$ in $\hull(\wt{X}_{\wt{v}})$.

Let $\mathcal{H}_{\wt{v}}$ denote the hyperplanes intersecting $\wt{X}_{\wt{v}}$.
Let $H \in \mathcal{H}_{\wt{v}}$.
The intersection $\wt{X}_{\wt{v}} \cap H$ is a geodesic line in $\wt{X}_{\wt{v}}$.
Let $g \in G_{\wt{v}}$ be an isometry that stabilizes an axis in $\wt{X}_{\wt{v}}$ that is not parallel to $\wt{X}_{\wt{v}} \cap H$.
Then $\{g^r H \}_{r\in \mathbb{Z}}$ is an infinite family of hyperplanes such that $\{g^r H \cap \wt{X}_{\wt{v}} \}_{r\in \mathbb{Z}}$ is a set of disjoint parallel lines in $\wt{X}_{\wt{v}}$.
As $\wt{Y}$ is finite dimensional, there exists an $N$ such that $H$ and $g^N H$ do not intersect.
Otherwise $\{g^r H \}_{r\in \mathbb{Z}}$ would be an infinite set of pairwise intersecting hyperplanes, which would imply that there are cubes of arbitrary dimension in $\wt{Y}$.

Therefore, as there are only finitely many $G_{\wt{v}}$-orbits of hyperplanes intersecting $\wt{X}_{\wt{v}}$, there exists a finite set of hyperplanes $H_1, \ldots, H_d \in \mathcal{H}_{\wt{v}}$ and $g_1, \ldots, g_d \in G$ such that $\mathcal{H}_{\wt{v}} = \{ g_1^r H_1, \ldots, g_d^r H_d \}_{r\in \mathbb{Z}}$, and each $\{ g_i^r H_i \}_{r\in \mathbb{Z}}$ is a disjoint set of hyperplanes in $\wt{Y}$.
Therefore $\{g_i^r H_i \cap \wt{E} \}_{r\in\mathbb{Z}}$ is a set of disjoint geodesic lines in $\wt{X}_{\wt{v}}$.
Thus, given a $0$-cube $y$, there exists a unique $y_i \in \mathbb{Z}$ such that $y[g_i^{y_i} H_i]$ and $y[g_i^{y_i + 1} H_i]$ properly intersect each other.
Therefore, construct $\phi: \hull(\wt{X}_{\wt{v}}) \rightarrow \mathbf{R}^d$ by letting $\phi(y) = (y_1, \ldots, y_d)$ for each $0$-cube $y$.
The map $\phi$ extends to the $1$-skeleton of $\hull(\wt{X}_{\wt{v}})$ since adjacent $0$-cubes lie on the opposite sides of precisely one hyperplane.
Therefore $\phi$ extends to the higher dimensional cubes, and thus $\hull(\wt{X}_{\wt{v}})$.
\end{proof}


Let $\wt{X} \rightarrow \wt{Y}$ be the lift of the universal cover of an amicable immersion $X \rightarrow Y$.
 Let $\wt{X}_{\wt{e}}$ be an edge space adjacent to a vertex space $\wt{X}_{\wt{v}}$.
 A hyperplane $H$ in $\wt{Y}$ intersects $\wt{X}_{\wt{v}}$ \emph{parallel} to $\wt{X}_{\wt{e}}$ if $H \cap \wt{X}_{\wt{v}}$ is a geodesic line parallel to $\wt{X}_{\wt{e}} \cap \wt{X}_{\wt{v}}$.
Otherwise, if $H \cap \wt{X}_{\wt{v}}$ is a geodesic line that is not parallel to $\wt{X}_{\wt{e}} \cap \wt{X}_{\wt{v}}$, then we say $H$ intersects $\wt{X}_{\wt{v}}$ \emph{non-parallel} to $\wt{X}_{\wt{e}}$.

\begin{lem} \label{lem:hyperplaneIntersections}
 Let $X \rightarrow Y$ be an amicable immersion.
 Let $\wt{e}$ be an edge in $\wt{\Gamma}$.
 Suppose that $H \subseteq \wt{Y}$ is a hyperplane intersecting $\wt{X}_{-\wt{e}}$ non-parallel to $\wt{X}_{\wt{e}}$, then $H$ intersects $\wt{X}_{+\wt{e}}$ non-parallel to $\wt{X}_{\wt{e}}$.
 Moreover, there is an arc in $H \cap \wt{X}_{\wt{e}}$ joining $H \cap \wt{X}_{-\wt{e}}$ to $H \cap \wt{X}_{+\wt{e}}$.
\end{lem}

 \begin{proof}
 Let $G = \pi_1 X$.
 The geodesic lines $H \cap \wt{X}_{-\wt{e}}$ and $\wt{X}_{\wt{e}} \cap \wt{X}_{-\wt{e}}$ are non parallel in $\wt{X}_{-\wt{e}}$, and therefore intersect in a single point $p \in \wt{X}_{\wt{e}} \cap \wt{X}_{-\wt{e}}$.
 As $H$ is two sided in $\wt{Y}$ and $X$ the vertex and edge spaces are transverse to $H$, the intersection of $H$ with $X$ is also locally two sided in $\wt{X}$.
 Therefore, $p$ is contained inside a curve in $H \cap \wt{X}_{\wt{e}}$.
 As $\wt{X}_{\wt{e}}$ is $G_{\wt{e}}$-invariant and only finitely many hyperplanes separate any two points in $\wt{X}$, we can deduce that $p$ is an endpoint of a compact curve in $H \cap \wt{X}_{\wt{e}}$ with its other endpoint contained in $\wt{X}_{+\wt{e}} \cap \wt{X}_{\wt{e}}$.
 Thus, $H$ must also intersect $\wt{X}_{+\wt{e}}$ non-parallel to $\wt{X}_{\wt{e}}$.
 \end{proof}

\begin{lem} \label{lem:fortifiedSpecialCase}
Let $X \rightarrow Y$ be an amicable immersion, where $Y$ is a finite dimensional, locally finite, nonpositively curved cube complex.
%
 If $\pi_1 X \cong \mathbb{Z}^2 *_{\mathbb{Z}} \mathbb{Z}^2$, then for every vertex space $\wt{X}_{\wt{v}}$ and adjacent edge space $\wt{X}_{\wt{e}}$ there is a hyperplane $H$ in $\wt{Y}$ that intersects $\wt{X}_{\wt{v}}$ parallel to $\wt{X}_{\wt{e}}$.
\end{lem}

\begin{proof}
Let $G = \pi_1 X$.
There are precisely two vertex orbits and one edge orbit in $\wt{\Gamma}$.
Assume that $\wt{Y} = \hull(\wt{X})$.
Let $\mathcal{H}$ denote the set of all hyperplanes in $\wt{Y}$ intersecting $\wt{X}$.
Let $\mathcal{H}_{\wt{v}}$ denote the set of all hyperplanes intersecting $\wt{X}_{\wt{v}}$.

For each vertex $\wt{v}$, there is precisely one $G_{\wt{v}}$-orbit of adjacent edges $\wt{e}$.
Therefore, if $H \in \mathcal{H}_{\wt{v}}$ is non-parallel to an adjacent edge space to $\wt{X}_{\wt{v}}$, it is non-parallel to all adjacent edge spaces, and by Lemma~\ref{lem:hyperplaneIntersections} it must intersect all adjacent vertex spaces non-parallel to all adjacent edge spaces.
Therefore, we deduce that any hyperplane that doesn't intersect every vertex space in $\wt{X}$ will either intersect a vertex space parallel to its adjacent edge spaces, or its intersection will be a line contained in an edge space.

 Suppose that there exists a vertex space $\wt{X}_{\wt{v}_1}$  such that no hyperplane in $\mathcal{H}_{\wt{v}_1}$ intersects $\wt{X}_{\wt{v}_1}$ parallel to the adjacent edge spaces.
Let $\hat{\mathcal{H}} = \mathcal{H} - \mathcal{H}_{\wt{v}_1}$.
 Every hyperplane in $\mathcal{H} - \mathcal{H}_{\wt{v}_1}$ must intersect each wall in $\mathcal{H}_{\wt{v}_1}$ so we deduce that $\wt{Y} = \hull(\wt{X}_{\wt{v}_1}) \times C(\wt{Y}, \hat{\mathcal{H}})$ (see~\cite[Lem 2.5]{CapraceSageev11}).
Furthermore, $\mathcal{H}_{g\wt{v}_1} = \mathcal{H}_{\wt{v}_1}$ for all $g\in G$.
By Lemma~\ref{lem:flatHull}, $\hull(\wt{X}_{\wt{v}})$ embeds in $\mathbf{R}^d$.
Since $\wt{X}_{\wt{v}_1}$ is contained inside some subcomplex $\hull(\wt{X}_{\wt{v}_1}) \times C \subseteq \wt{Y}$, where $C$ is the $0$-cube determined by orienting all hyperplanes towards $\wt{X}_{\wt{v}_1}$, we can conclude that only finitely many hyperplanes in $\hat{\mathcal{H}}$ intersect the $r$-neighborhood of $\wt{X}_{\wt{v}_1}$.

Let $\wt{X}_{\wt{u}}$ be an vertex space adjacent to $\wt{X}_{\wt{v}_1}$, and let $\wt{X}_{\wt{e}_1}$ be the edge space connecting them.
Let $\wt{X}_{\wt{v}_2}$ be another vertex space adjacent to $\wt{X}_{\wt{u}}$ and let $\wt{X}_{\wt{e}_2}$ be the edge space connecting them.
Note that $\wt{X}_{\wt{v}_1}$ and $\wt{X}_{\wt{v}_2}$ are in the same $G_{\wt{u}}$-orbit.
The geodesic lines $\wt{X}_{\wt{u}} \cap \wt{X}_{\wt{e}_1}$ and $\wt{X}_{\wt{u}} \cap \wt{X}_{\wt{e}_2}$ are parallel in $\wt{X}_{\wt{u}}$.
Let $D \subseteq \wt{X}_{\wt{u}}$ be the subspace isometric to $\mathbb{R} \times [a,b]$ bounded by these parallel lines.
Let $U = D \cup \wt{X}_{\wt{e}_1} \cup \wt{X}_{\wt{e}_2}$.
Finitely many hyperplanes in $\hat{\mathcal{H}}$ intersect $D$.

Let $g_1$ be an isometry in $G_{\wt{v}_1}$ that stabilizes an axis in $\wt{X}_{\wt{v}_1}$ that is non-parallel to the geodesic $\wt{X}_{\wt{v}_1} \cap \wt{X}_{\wt{e}_1}$.
Similarly, let $g_2$ be an isometry in $G_{\wt{v}_2}$ that stabilizes an axis in $\wt{X}_{\wt{v}_2}$ that is non-parallel to the geodesic $\wt{X}_{\wt{v}_2} \cap \wt{X}_{\wt{e}_2}$.
Note that $F = \langle g_1, g_2 \rangle$ is a free group on two generators.
Let $r$ be such that $U$ is contained in the $r$-neighbourhood of $\wt{X}_{\wt{v}_1}$.
As there are only finitely many hyperplanes in $\hat{\mathcal{H}}$ intersecting $\neb_r(\wt{X}_{\wt{v}_1})$, there must exist an $n$ such that $g_1^n$ stabilizes those walls.
Similarly, since $\wt{X}_{\wt{v}_2}$ is a translate of $\wt{X}_{\wt{v}_1}$, we can deduce that there are only finitely many hyperplanes in $\hat{\mathcal{H}}$ intersecting $\neb_r(\wt{X}_{\wt{v}_2})$, and there must exist an $m$ such that $g_2^m$ stabilizes those walls.
Let $F' = \langle  g_1^n, g_2^m \rangle$.
As $U$ lies in both $\neb_r(\wt{X}_{\wt{v}_1})$ and $\neb_r(\wt{X}_{\wt{v}_2})$ we can deduce that the hyperplanes in $\hat{\mathcal{H}}$ that intersect the $F'$-translates of $U$ are precisely the hyperplanes intersecting $U$.
Let $\wt{Z} = F' \wt{X}_{\wt{v}_1} \cup F' \wt{X}_{\wt{v}_2} \cup F' U$.
Then $\hull(\wt{Z}) = \hull(\wt{X}_{\wt{v}_1}) \times K \subseteq \wt{Y}$, where $K$ is a compact cube complex.
Then $F$ acts freely on $\hull(\wt{Z})$, which is a contradiction since number of $0$-cubes intersecting the $r$-neighborhood of a $0$-cube in $\hull(\wt{X}_{\wt{v}_1}) \times K $ grows polynomially with $r$, and therefore cannot permit a free $F$-action.
\end{proof}

\noindent Lemma~\ref{lem:fortifiedSpecialCase} is a special case of the following more general statement:

\begin{cor} \label{cor:fortifiedGeneralCase}
 Let $X \rightarrow Y$ be an amicable immersion, where $Y$ is a finite dimensional, locally finite, nonpositively curved cube complex.
%
 Then for every vertex space $\wt{X}_{\wt{v}}$ and adjacent edge space $\wt{X}_{\wt{e}}$ there is a hyperplane $H$ in $\wt{Y}$ that intersects $\wt{X}_{\wt{v}}$ parallel to $\wt{X}_{\wt{e}}$.
\end{cor}

\begin{proof}
For every edge $\wt{e}$ in $\wt{\Gamma}$ there is a subgroup $G' = \langle G_{-\wt{e}}, G_{+\wt{e}} \rangle \leq G$ such that $G' \cong \mathbb{Z}^2 *_{\mathbb{Z}} \mathbb{Z}^2$.
Let $Y' = G' \backslash \wt{Y}$.
Then there is an amicable immersion $X' \rightarrow Y'$ such that $\wt{X}'_{\wt{e}} = \wt{X}$ and $\wt{X}'_{\pm \wt{e}} = \wt{X}_{\pm \wt{e}}$.
Therefore, by Lemma~\ref{lem:fortifiedSpecialCase}, there is a hyperplane intersecting $\wt{X}_{-\wt{e}}$ parallel to $\wt{X}_{\wt{e}}$, and similarly for $\wt{X}_{+\wt{e}}$.
\end{proof}

\noindent The following proposition is a strengthening of one direction of Theorem 1.1 in~\cite{Wise13}.
Let $f_1: A \rightarrow C$ and $f_2: B \rightarrow C$ be maps between topological spaces $A,B,C$.
The \emph{fiber product} $A \otimes_C B = \{ (a,b)\in A\times B \mid f_1(a) = f_2(b) \}$.
Note that there are natural projections $p_1: A \otimes_C B \rightarrow A$ and $p_2: A \otimes_C B \rightarrow B$.

\begin{prop} \label{prop:inheritFiniteDim}
 Let $G$ be a tubular group acting freely on a CAT(0) cube complex $\wt{Y}$.
 Then there is a tubular space $X$ with a finite set of immersed walls such that the associated wallspace $(\wt{X}, \mathcal{W})$ has the following properties:
 \begin{enumerate}
    \item $G$ acts freely on $C(\wt{X}, \mathcal{W})$.
 	\item \label{claim:finDim}  $C(\wt{X}, \mathcal{W})$ is finite dimensional if $\wt{Y}$ is finite dimensional.
 	\item \label{claim:locFin} $C(\wt{X}, \mathcal{W})$ is finite dimensional and locally finite if $\wt{Y}$ is locally finite.
 \end{enumerate}
\end{prop}

\begin{proof}

Let $Y = G \backslash \wt{Y} $.
Let $X \rightarrow Y$ be an amicable immersion.
Assume that $\wt{Y} = \hull(\wt{X})$, so every immersed hyperplane in $ Y$ intersects $ X$.
Therefore, as $X$ is compact, there are finitely many immersed hyperplanes $h_1, \ldots h_m$ in $Y$.

Let $h_i \rightarrow Y$ be an immersed hyperplane in $Y$.
We obtain horizontal immersed walls in $X$ by considering the components of the fiber product $X \otimes_Y h_i$ of $X \rightarrow Y$ and $h \rightarrow Y$.
Each component $\Lambda$ has a natural map into $X$.
The components of $X \otimes_Y h$ that have image in $X$ contained in an edge space are ignored.
Let $\Lambda^p$ be a component of $X \otimes_Y h$ whose image in $X$ intersects a vertex space $X_v \subseteq X$.
We will show that after a minor adjustment to $\Lambda^p$, we obtain a horizontal immersed wall and by considering all such components we obtain a set of horizontal walls in $X$ obtained from an equitable set.

Using the map $\Lambda^p \rightarrow X$ we can decompose $\Lambda^p$ into the components of the preimages of vertex space and edge spaces.
As the intersection of each hyperplane $H\subseteq \wt{Y}$ with each vertex space $\wt{X}_{\wt{v}}$ is either empty or a geodesic line, the intersection of each $h_i$ with $X_v$ is a set of geodesic curves, so $\Lambda^p$ restricted to the preimage of $X_v$ is a set of geodesic curves.
By Lemma~\ref{lem:hyperplaneIntersections} each hyperplane $H \subseteq \wt{Y}$ that intersects a vertex space $X_{\wt{v}}$ non-parallel to an adjacent edge space $\wt{X}_{\wt{e}}$ will intersect $\wt{X}_{\wt{e}}$ as an arc with endpoints in $\wt{X}_{-\wt{e}}$ and $\wt{X}_{+\wt{e}}$.
Thus the components of the intersection $X_e \cap h_i$ that intersect $X_{-e}$ or $X_{+e}$, are arcs with endpoints in both $X_{-e}$ and $X_{+e}$.
Therefore, $\Lambda^p$ decomposes into circles that map as local geodesics into vertex spaces, and arcs that map into edges spaces $X_e$ with an endpoint in each $X_{-e}$ and $X_{+e}$.

Let $\{\Lambda_1^p, \ldots, \Lambda_k^p\}$ be the set of all such components of $X \otimes_Y h_i$ that intersect vertex spaces.
Let $S_v^p$ be the set curves that map the circles in $\{\Lambda_i^p \}^k_{i=1}$ to the vertex space $X_v$.
The elements of $S_v^p$ and the attaching maps $\varphi_e^{\pm}$ of the edge spaces in $X$ are locally geodesic curves, and $\#[\varphi_e^+, S_{e^+}] = \#[\varphi_e^-, S_{e^-}]$ since both sides are equal to the number of arcs in the walls $\{ \Lambda_i^p \}^k_{i=1}$ that map into $X_e$.
 As $G$ acts freely on $\wt{Y}$, there must be hyperplanes intersecting each vertex space $\wt{X}_{\wt{e}}$ as geodesics in at least two parallelism classes.
This implies that $S_v^p$ contains curves generating at least two non-commensurable cyclic subgroups of $G_v$, and therefore $S_v^p$ generates a finite index subgroup of $G_v$.

 $S_v^p$ is almost an equitable set: the images of the curves in $S_v^p$ may not be disjoint.
 Suppose that $\alpha_1, \ldots \alpha_m \in S_v$ be a maximal set of curves that have identical image in $X_v$.
 Let $\neb_\epsilon(Q)$ denote the $\epsilon$-neighborhood of a subset $Q$ of either $Y$ or $\wt{Y}$ with respect to the CAT(0) metric.
Let $\epsilon \in (0, \frac{1}{3})$ be such that the neighbourhood $\neb_{\epsilon}(\alpha_1) \subseteq Y$ only contains the images of $\alpha_1, \ldots, \alpha_m$ and the arcs connected to them.
There is a homotopy of $\bigsqcup_{i=1}^k \Lambda_i^p \rightarrow X$ that is the identity outside of $\bigsqcup_{i=1}^k \Lambda_i^p \cap N_{\epsilon}(\alpha_1)$ such that $\alpha_1, \ldots, \alpha_m$ are homotoped to a disjoint set of geodesic curves in $X_v \cap N_{\epsilon}(\alpha_1)$ transverse or disjoint from all the other curves in $S_v^p$.
By choosing $\epsilon$ small enough we can perform such a homotopy $\Phi: \bigsqcup_{i=1}^k \Lambda_i^p \times [0,1] \rightarrow X$ such that all sets of overlapping curves in $\{S_{v}^p\}_{v\in V\Gamma}$ become disjoint and such that $\Phi$ is the identity map outside of the $\epsilon$-neighborhood of the overlapping curves.
The restriction of $\Phi$ to $\Lambda_i^p \times \{1 \} \rightarrow X$ is an immersed wall that we will denote by $\Lambda_i$.
Thus, the immersed walls $\{ \Lambda_i \}_{i=1}^k$ obtained from an equitable set $\{ S_v \}_{v \in V\Gamma}$.
We refer to $\{ \Lambda_i^p \}_{i=1}^k$ as the \emph{immersed proto-walls} and the lifts $\wt{\Lambda}_i^p \rightarrow \wt{X}$ as the \emph{proto-walls}.
Note that proto-walls have regular and non-regular intersections in the same way that walls do.

Let $(\wt{X}, \mathcal{W})$ be the wallspace obtained from the immersed walls $\{ \Lambda_i \}_{i=1}^k$ and adding a single vertical wall for each edge space.
Each wall $\wt{\Lambda} \rightarrow \wt{X}$ covers an immersed wall $\Lambda \rightarrow X$.
There exists a homotopy of $\Lambda \rightarrow X$ to the corresponding immersed proto-wall $\Lambda^p \rightarrow X$.
This homotopy lifts to a homotopy from the immersed wall $\wt{\Lambda} \rightarrow \wt{X}$ to a unique proto-wall $\wt{\Lambda}^p \rightarrow \wt{X}$.
Note that each wall is contained in the $\epsilon$-neighborhood of its corresponding proto-wall.
Each proto-wall corresponds to the intersection of a unique hyperplane in $\wt{Y}$ with the image of $\wt{X}$ in $\wt{Y}$.
 Therefore, each wall in $\mathcal{W}$ corresponds to a unique hyperplane in $\wt{Y}$.

 Let $\wt{\Lambda}$ be a wall in $\mathcal{W}$, and let $\wt{\Lambda}^p$ be the corresponding proto-wall.
Note that $\wt{\Lambda} \cap \wt{X}_{\wt{v}}$ and $\wt{\Lambda}^p \cap \wt{X}_{\wt{v}}$ are either parallel geodesic lines, or both empty intersections.
Therefore, if $\wt{\Lambda}_1, \wt{\Lambda}_2 \in \mathcal{W}$ are a pair of regularly intersecting walls, then they correspond to a pair of regularly intersecting proto-walls, which correspond to a pair of intersecting hyperplanes in $\wt{Y}$.

If a pair of proto-walls $\wt{\Lambda}_1^p$ and $\wt{\Lambda}_2^p$ are disjoint, then the corresponding walls in $\wt{\Lambda}_1$ and $\wt{\Lambda}_2$ in $\mathcal{W}$ are also disjoint.
Moreover, since $\wt{\Lambda}$ is contained in the $\epsilon$-neighborhood of $\wt{\Lambda}^p$, a halfspace of $\wt{\Lambda}$ determines a halfspace of $\wt{\Lambda}^p$ and therefore a halfspace of the hyperplane $H$ corresponding to $\wt{\Lambda}^p$.

To prove (\ref{claim:finDim}), suppose that $C(\wt{X}, \mathcal{W})$ were infinite dimensional, then by Proposition~\ref{prop:InfDimInfCube}, there would exists an infinite set of pairwise regularly intersecting walls in $\mathcal{W}$, which implies there is an infinite set of pairwise regularly intersecting proto-walls.
 Therefore, there is an infinite set of pairwise intersecting hyperplanes in $\wt{Y}$.
This would imply that $\wt{Y}$ is an infinite dimensional CAT(0) cube complex.
Therefore, if $\wt{Y}$ is finite dimensional, then so is $C(\wt{X},\mathcal{W})$.

To prove (\ref{claim:locFin}) we first prove the following:

\begin{claim}
 If $\wt{Y}$ is locally finite, then $C(\wt{X},\mathcal{W})$ is finite dimensional.
\end{claim}
\begin{proof}
Suppose that $\wt{Y}$ is locally finite.
If $C(\wt{X}, \mathcal{W})$ is infinite dimensional, then by Lemma~\ref{prop:InfDimInfCube} it contains an infinite cube containing a canonical $0$-cube $z$.
Let $\wt{\Lambda}_1, \ldots, \wt{\Lambda}_n, \ldots$ be the set of infinite pairwise crossing walls corresponding to the infinite cube.
Let $\wt{\Lambda}_1^p, \ldots, \wt{\Lambda}_n^p, \ldots$ be the corresponding set of infinite pairwise crossing proto-walls, and let $H_1, \ldots, H_n, \ldots$ be the corresponding infinite family of pairwise crossing hyperplanes.

Suppose that $Q$ is a subcomplex in $\wt{Y}$.
 Let $U(Q)$ denote the \emph{cubical neighborhood} of $Q$, which is the union of all cubes in $U(Q)$ that intersect $Q$.
As $\wt{Y}$ is locally finite, if $Q$ is compact, then $U(Q)$ is also compact.
By~\cite[Lem 13.15]{HaglundWise08}, if $Q$ is convex, then so is $U(Q)$.
Let $U^n(Q)$ denote the cubical neighborhood of $U^{n-1}(Q)$.

Let $x \in \wt{X}$ be a point determining the canonical $0$-cube $z$ in $C(\wt{X},\mathcal{W})$.
Let $x$ be contained in a cube $C$ in $\wt{Y}$.
As $C$ is compact and convex, $U^n(C)$ is also compact and convex, and therefore can only be intersected by finitely many $H_i$.
There exists an $H_i$ such that $H_i$ intersects $U^{N+2}(C)$, but not $U^{N+1}(C)$ for some $N>1$.
Since $H_i$ does not intersect $U^{N}(C)$ nor $U^{N+1}(C)$, there must exist a hyperplane $H$ intersecting $U^{N+1}(C)$ that separates $U^{N}(C)$ from $H_i$.
Note that $\dist_{\wt{Y}}(x, H) \geq N$ and $\dist_{\wt{Y}}(H, H_i) \geq 1$.
Let $\wt{\Lambda}^p$ be the proto-wall corresponding to $H$ and $\wt{\Lambda}$ be the corresponding wall.
As $\wt{\Lambda}^p$ separates $x$ from $\wt{\Lambda}_i^p$, we can conclude $\wt{\Lambda}$ separates $x$ from $\wt{\Lambda}_i$ since the $\wt{\Lambda}$ and $\wt{\Lambda}_i$ are respectively contained in the $\epsilon$-neighborhoods of $\wt{\Lambda}^p$ and $\wt{\Lambda}_i^p$.
This contradicts the fact that $z$ is incident to a $1$-cube dual to hyperplane corresponding to $\wt{\Lambda}_i$.
\end{proof}

As $C(\wt{X},\mathcal{W})$ is finite dimensional, we can apply Corollary~\ref{cor:fortifiedGeneralCase} to each edge group in $G$ to deduce that $\{\Lambda_i\}_{i=1}^k$ are fortified.
Therefore, by Proposition~\ref{prop:fortified} we deduce that $C(\wt{X},\mathcal{W})$ is locally finite.
\end{proof}

We can now prove the main theorem of this paper.

\begin{thm} \label{mainA}
 A tubular group $G$ acts freely on a locally finite CAT(0) cube complex if and only if $G$ is virtually special.
\end{thm}

\begin{proof}
 Suppose that $G$ is virtually special.
 Then $G$ embeds as the subgroup of a finitely generated right angled Artin group, and therefore acts freely on the universal cover of the corresponding Salvetti complex, which is necessarily locally finite.

 Conversely, suppose that $G$ acts freely on a locally finite CAT(0) cube complex.
 Let $X$ be a tubular space such that $G = \pi_1X$.
 By Proposition~\ref{prop:inheritFiniteDim} there exists a finite set of immersed walls such that the dual of the associated wallspace $C(\wt{X},\mathcal{W})$ is finite dimensional and locally finite.
 By Lemma~\ref{lem:primitive} we can assume that the immersed walls are also primitive.
 Therefore, by Proposition~\ref{prop:virtually special}, $G$ is virtually special.
\end{proof}


\section{Virtual Cubical Dimension} \label{section:VirtualCubicalDimension}

\begin{lem} \label{firstHomSurvival}
Let $X$ be a tubular space and $G = \pi_1 X$.
Suppose there exists an equitable set that produces primitive, non-dilated immersed walls in $X$.
 There exists a finite index subgroup $G' \leqslant G$ such that for each vertex group $G'_v$ of the induced splitting of $G'$ the natural map $\homology_1(G_v') \rightarrow \homology_1(G')$ is an injection as a summand.
\end{lem}

\begin{proof}
 Note that $\homology_1(G)$, and its finite index subgroups, is a summand of two factors.
 The first factor $\homology_1^{\textsf{v}}(G)$ is generated by the image of the vertex groups of $G$, and the second factor $\homology_1^{\textsf{v}}(G)$ is generated by the stable letters in the graph of groups presentation.

Let $\wt{\Gamma}$ be the Bass-Serre tree.
 By Proposition~\ref{prop:RxT}, since there are immersed walls that are primitive and non-dilated, $G$ acts freely on $\mathbf{R}^d \times \wt{\Gamma}$ such that $G_{\wt{v}}$ fixes the vertex $\wt{v}$ in $\wt{\Gamma}$.
 Therefore $G $ is a subgroup of $\Aut(\mathbf{R}^d \times \wt{\Gamma}) \cong \mathbb{Z}^d \rtimes \Aut([-1,1]^d) \times \Aut(\wt{\Gamma})$.
 There is a finite quotient $\mathbb{Z}^d \rtimes \Aut([-1,1]^d) \times \Aut(\wt{\Gamma}) \rightarrow \Aut([-1,1]^d)$, so let $G'' \leqslant G$ be the finite index subgroup contained in the kernel.
 Note that $G''$ embeds in $\mathbb{Z}^d \times \Aut(\wt{\Gamma})$.
 Let $p: G'' \rightarrow \mathbb{Z}^d$ be the projection onto the first factor.
 Each vertex group survives in the image of $p$, and therefore we have embedding $\homology_1(G_v'') \hookrightarrow \homology_1^{\textsf{v}}(G'') \hookrightarrow \mathbb{Z}^d$.

 For each vertex $v$ in $\Gamma'' = \wt{\Gamma} \slash G''$ there is a finite index subgroup $A_v \leqslant \mathbb{Z}^d$ such that $p(G_v'')\cap A_v$ is a summand of $A_v$.
 Let $A = \cap_v A_v$ and $G' = p^{-1}(A)$.
 Each vertex group in $G'$ will be a factor in $A$.
 As $A$ is free abelian, the map $G' \rightarrow A$ will factor through the $\homology_1^{\textsf{v}}(G')$, so we can deduce that each vertex group survives as a retract in $\homology_1(G')$.
 Therefore, each vertex group in $G'$ survives as a summand in the first homology.
%
 \end{proof}

 \begin{thm} \label{mainC}
   Let $G$  be a tubular group acting freely on a finite dimensional CAT(0) cube complex.
    Then there is a finite index subgroup $G' \leqslant G$ such that $G'$ acts freely on a 3-dimensional CAT(0) cube complex.
 \end{thm}

 \begin{proof}
 Let $X$ be a tubular space such that $G = \pi_1 X$.
 By Proposition~\ref{prop:inheritFiniteDim} there exists immersed walls in $X$ that are non-dilated, and by Lemma~\ref{lem:primitive} we can assume that they are also primitive.
 Let $G' \leqslant G$ be the finite index subgroup given by Lemma~\ref{firstHomSurvival}, and let $X'$ be the corresponding covering space.
   Let $\homology_1^{\textsf{v}}(G') \cong \mathbb{Z}^d$ be the summand in the first homology generated by the vertex groups.
  By Lemma \ref{firstHomSurvival}, there is an inclusion $\iota_v : G'_v \rightarrow \homology^{\textsf{v}}_1(G')$, and a projection map $p_v : \homology_1^{\textsf{v}}(G') \rightarrow G'_{v}$.

 Suppose that $d = 2$.
 Choose any pair of elements $a,b \in \homology_1^{\textsf{v}}(G')$ that generate $\homology_1^{\textsf{v}}(G')$.
 We claim that $S = \{ S_v = \{ p_v(a), p_v(b) \} \mid v \in V(\Gamma) \}$ is an equitable set.
 By construction, each $S_v$ generates $G_v'$.
 The edge group $G_e' = \la g_e \ra$ is adjacent to $G_{-e}'$ and $G_{+e}'$ and the respective inclusions are $(\varphi^-_e)_*$ and $(\varphi^+_e)_*$.
 There is an isomorphism $A = p_{+e} \circ \iota_{-e} \in SL(2, \mathbb{Z})$, that maps $(\varphi^-_e)_*(g_e)$ in $G_u'$ to $(\varphi^+_e)_*(g_e)$ in $G_v'$.
 Therefore,
 \[
  \#[p_u(a), (\varphi^-_e)_*(g_e)] = \#[Ap_u(a), A(\varphi^-_e)_*(g_e)] = \#[p_v(a), (\varphi^+_e)_*(g_e)],
 \]
 A congruent set of equalities exist for $b$.
 These equalities also imply that the choice of arcs for the equitable sets can be chosen so that they only join circles that are the image of the same elements of $\homology_1^{\textrm{v}}(G')$. Therefore a set of embedded immersed walls is obtained with precisely two immersed walls intersecting each vertex space: such a set of horizontal immersed walls along with a vertical wall for each edge will give a three dimensional dual cube complex.

 If $d > 2$, then there exist vertex groups $G_u$ and $G_v$ such that they embed into $\homology_1^{\textsf{v}}(G')$ as distinct summands.
 Let $G_u = \langle g_u, g_u' \rangle$ and $G_v = \langle g_v, g_v' \rangle$.
 We can assume, since they are distinct summands that $g_u$ is disjoint from the image of $G_v$ in $\homology_1^{\textsf{v}}(G')$, and that $g_v$ is disjoint from the image of $G_u$ in $\homology_1^{\textsf{v}}(G')$.
  By attaching an edge space in $X'$ connecting $X_u$ and $X_v$, that have attaching maps representing $g_u$ and $g_v$ respectively we obtain a new graph of spaces $\hat{X}'$.
 The resulting tubular group $\hat{G}' = \pi_1 \hat{X}'$ has $\homology_1^{\textsf{v}}(\hat{G}') \cong \mathbb{Z}^{d-1}$, so by induction we can obtain the specified graph of spaces $\hat{X}'$.
 Given a set of immersed walls for $\hat{X}'$ with a dual cube complex of dimension at most 3 we obtain immersed walls for $X'$ by deleting the arcs that map into the the edge space that was added to construct $\hat{X}'$.
 The immersed walls obtained still give a dual cube complex with dimension at most 3.
\end{proof}

\bibliographystyle{plain}
\bibliography{Ref}
\end{document}